\documentclass[12pt]{amsart}
\usepackage{amsmath,amssymb, mathtools,a4wide} 
\usepackage{amscd} 
\usepackage{amsthm} 
\usepackage{dsfont} 
\usepackage{MnSymbol} 

\usepackage{float} 

\usepackage[pdftex]{graphicx} 
\usepackage{xcolor} 
\usepackage[colorlinks,citecolor=blue,linkcolor=blue,urlcolor=blue]{hyperref} 

\usepackage[utf8]{inputenc} 

\usepackage{wrapfig} 
\usepackage{multicol} 

\usepackage{tikz}
\usepackage{mathrsfs}
\usetikzlibrary{arrows}
\usetikzlibrary{intersections, calc}

\newtheorem{thm}{Theorem}

\newtheorem{lemma}{Lemma}
\newtheorem{coro}{Corollary}

\newtheorem{prop}{Proposition}
\newtheorem*{example}{Example}
\newtheorem*{Remark}{Remark}

\numberwithin{equation}{section} 

\newcommand*{\R}{\mathbb{R}}
\newcommand*{\C}{\mathbb{C}}

\renewcommand*{\l}{\lambda}

\newcommand*{\mc}{\mathcal}

\begin{document}
\author{Alexander Thomas}
\title[Cross ratio in a pencil of conics]{Geometric characterizations of the cross ratio in a pencil of conics}
\address{Max-Planck Institute for Mathematics, Vivatsgasse 7, 53111 Bonn, Germany}
\email{athomas@mpim-bonn.mpg.de}

\begin{abstract}
We give geometric ways to determine the cross ratio of conics in a pencil using elementary projective geometry.
\end{abstract}

\maketitle

\vspace*{-0.6cm}
\begin{figure}[h]
\centering
\includegraphics[height=3.8cm]{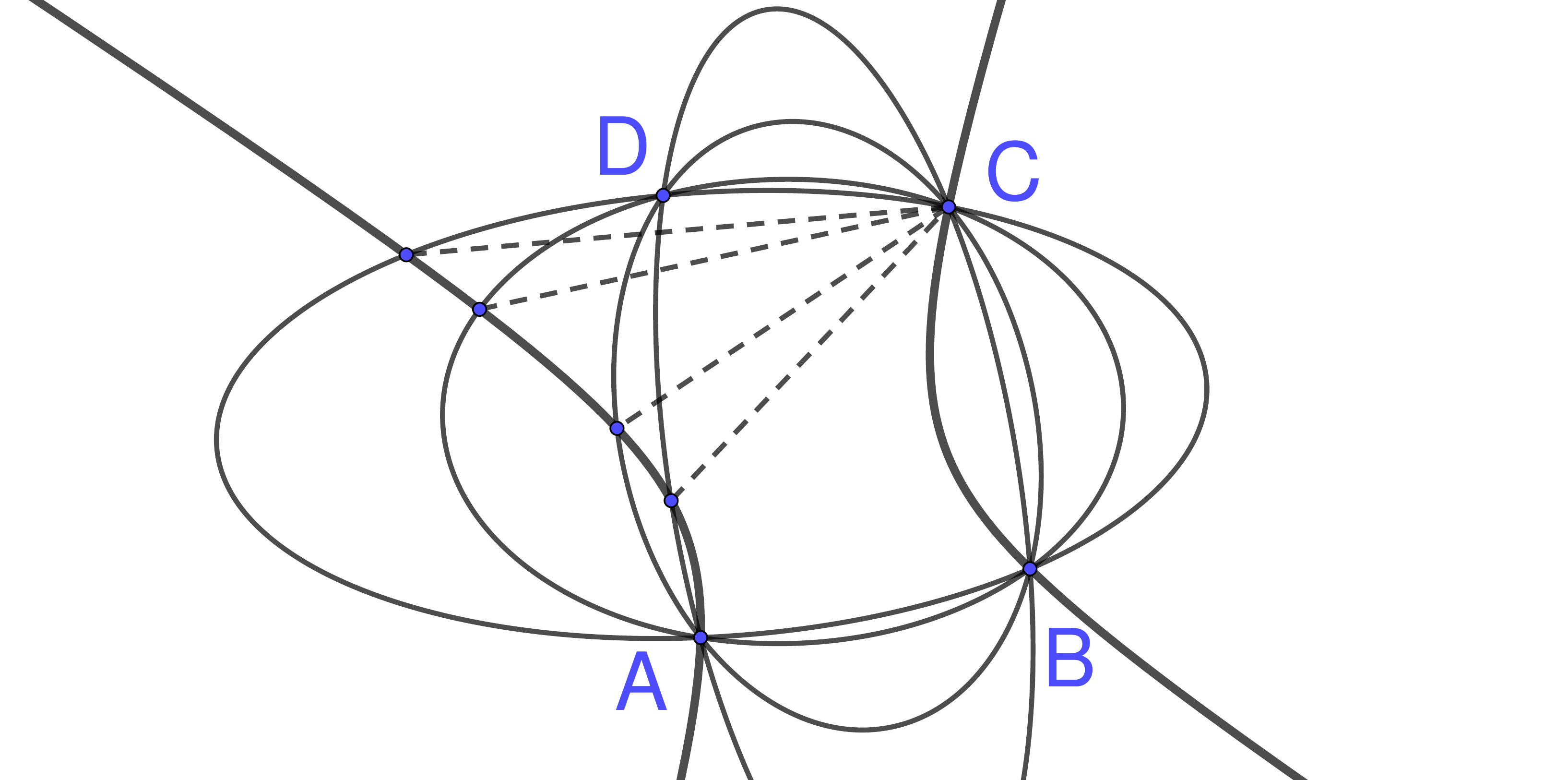}
\end{figure}

\vspace*{-0.5cm}

\section*{Introduction}

In projective geometry the fundamental invariant is the cross ratio. On a line, the action of projective transformations is 3-transitive, i.e. a triple of points can be mapped to any other triple of points via a projective transform. For four points, the cross ratio is an invariant.

On the plane, the cross ratio is defined for four points on a common line, and by duality for four lines going through a common point. Collinear points and concurrent lines are examples of pencils. So the question raises \emph{whether there is a notion of cross ratio in a pencil of conics}.

Abstractly, there is the following answer: denote by $Q(\C^2)$ the space of quadratic forms on $\C^2$, which can be considered as the space of conics. A pencil of conics is precisely a line in the projective space $\mathbb{P}(Q(\C^2))$, see \cite{Berger}, chapter 14. So there is a notion of cross ratio in a pencil of conics.

The question we address here is how to determine \emph{geometrically} the cross ratio of four conics in a pencil.
After summarizing basic facts from elementary projective geometry in section \ref{section1}, we give six geometric characterizations for the cross ratio of conics in section \ref{section2}.
Notice that we work in $\C P^2$, but our pictures are drawn in $\R^2$. We work only with non-degenerate pencils of conics, the other cases can be easily obtained as limit cases.

\textit{\textbf{References.}} Though conics were discovered about 2300 years ago, only very recently books emphasizing pictures and geometric approaches become available. Our main reference is the excellent book of Jean-Claude Sidler \cite{Sidler}. We also warmly recommend the books \cite{Akopyan}, \cite{Glaeser} and \cite{Kendig}.
Classical references include \cite{Berger} and \cite{Hilbert}. We recommend \cite{Michel} for a glance on the deep geometric knowledge at the beginning of the 20\textsuperscript{th} century.

\section{Basic Facts on the cross ratio}\label{section1}

We present here some facts on cross ratios in plane projective geometry which we will use in the sequel. For less known results we include proofs. Nothing in this section is new.

\subsection{First properties}

The cross ratio is denoted by $(A,B,C,D)$. We use here the convention 
$$(A,B,C,D)= \frac{(A-C)(B-D)}{(B-C)(A-D)}.$$

\begin{prop}\label{cross-ratio-line}
Let $(a,b,c,d)$ be four lines in a pencil (see Figure \ref{birapport-droites}). Any line not in the pencil intersects them in points $A, B, C, D$ respectively. Then $$(a,b,c,d)=(A,B,C,D).$$
\end{prop}

\begin{figure}[h]
\centering
\includegraphics[height=3.5cm]{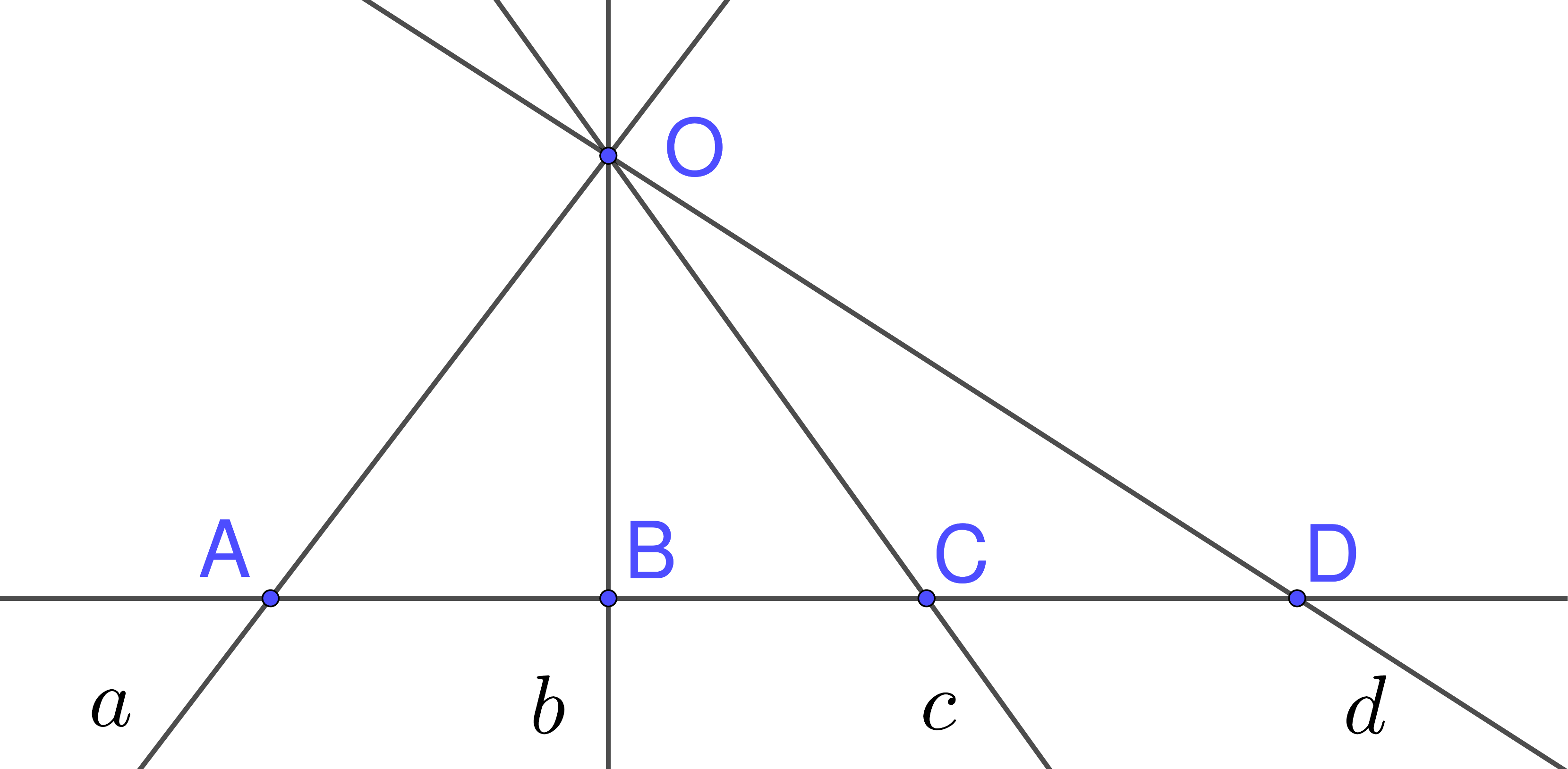}

\caption{Cross ratio of four lines}
\label{birapport-droites}
\end{figure}

\begin{Remark}
This proposition gives the fundamental invariance of the cross ratio under elementary projective transformations. 
The cross ratio of four concurrent lines can be expressed using the angles between the lines. Using the notation $\sphericalangle(a,b)$ for the angle between lines $a$ and $b$, we have $$(a,b,c,d) = \frac{\sin \sphericalangle(a,c) \;\; \sin \sphericalangle (b,d)}{\sin \sphericalangle (b,c) \;\; \sin \sphericalangle(a,d)}.$$
\end{Remark}

Conics have the following fundamental property concerning the cross ratio:
\begin{prop}\label{cross-ratio-conic}
Let $\mc{C}$ be a conic and $A, B, C, D$ be four points on it (see Figure \ref{birapport-on-conic}). Then for any point $P$ on $\mc{C}$, the cross ratio $(PA,PB,PC,PD)$ is independent of $P$.
\end{prop}

\vspace{-0.2cm}
\begin{figure}[h]
\centering
\includegraphics[height=3.5cm]{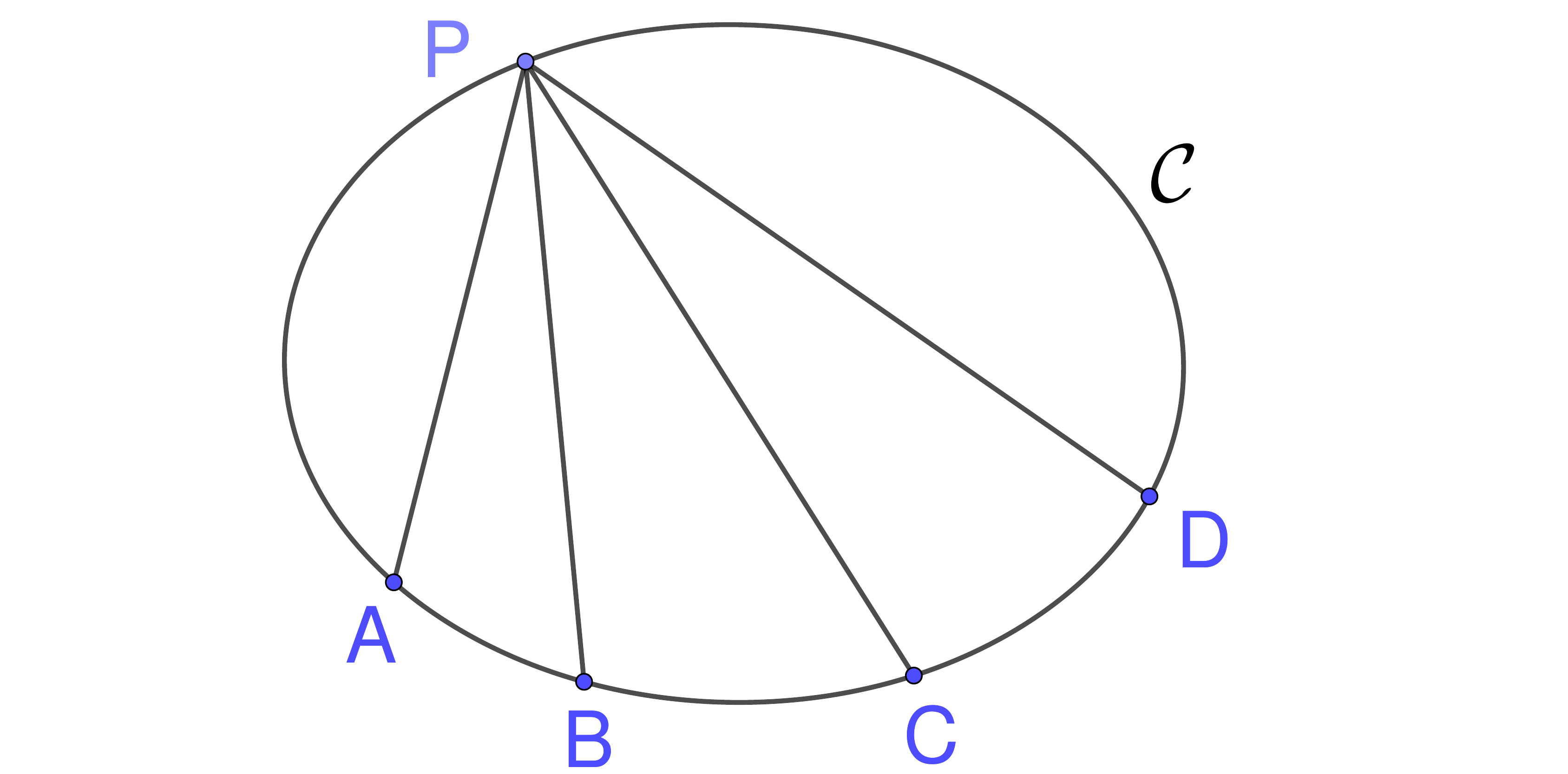}

\vspace{-0.2cm}
\caption{Cross ratio on a conic}
\label{birapport-on-conic}
\end{figure}

We denote this cross ratio of four points on a conic by $(A,B,C,D)_{\mc{C}}$. The previous proposition can actually be seen as a definition of a conic from five given points $A, B, C, D$ and $P$: the set of points $Q$ such that $(QA,QB,QC,QD) = (PA,PB,PC,PD)$.

\begin{wrapfigure}{r}{0.5\textwidth}
\centering
   \includegraphics[width=0.37\textwidth]{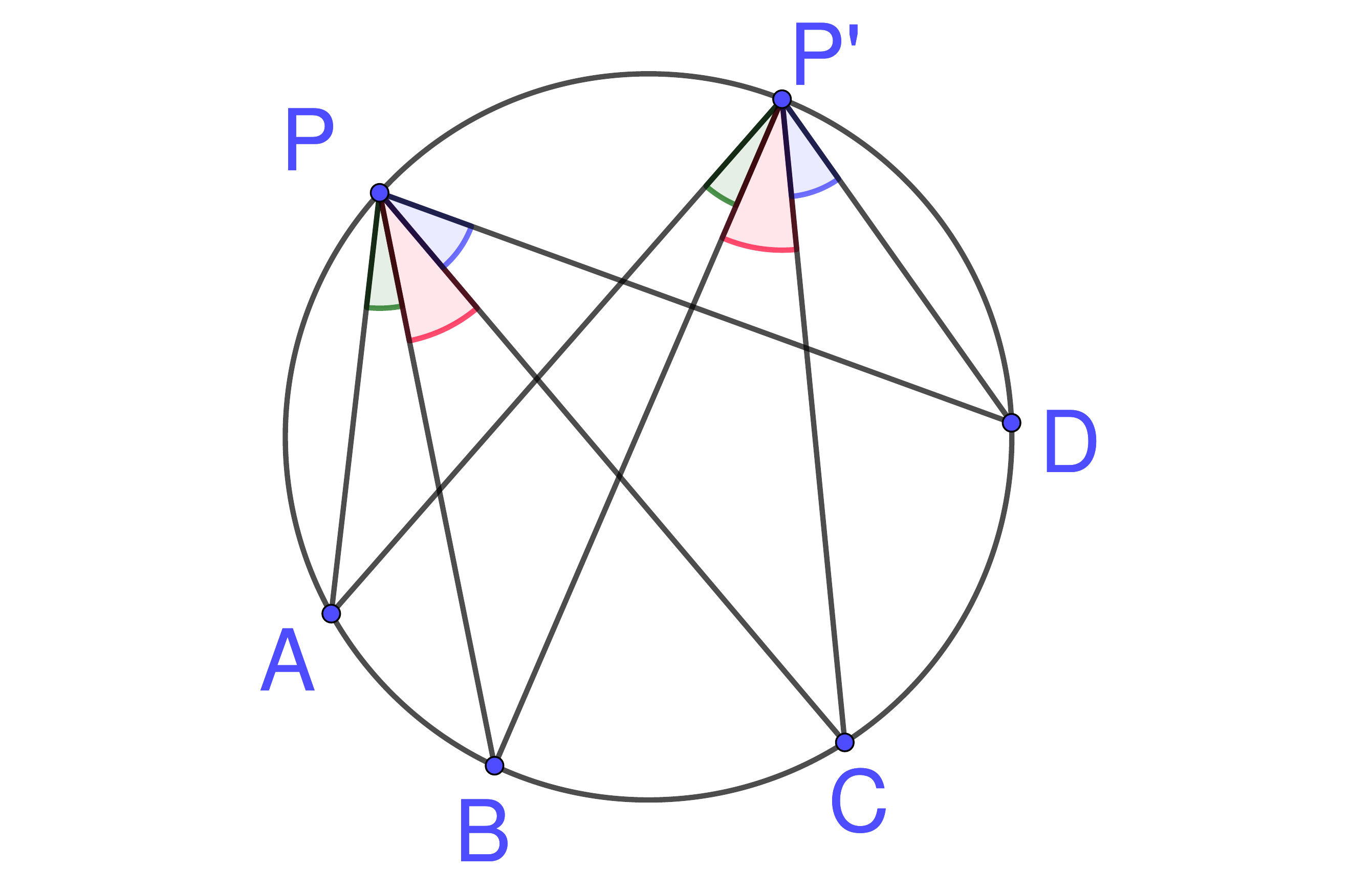}
 
\vspace{0.4cm}
\caption{Angle equalities on the circle}	
\label{angles-circle}
\vspace{-0.5cm}
\end{wrapfigure}

\vspace{0.3cm}
\textit{Proof.} Among the multiple equivalent definitions of a conic, take the following one: a conic is the image of a circle under a projective transformation.

For a circle, we know that the cross ratio $(PA,PB,PC,PD)$ is independent of $P$ since all the angles between the four lines $PA, PB, PC$ and $PD$ are independent of $P$ (angle property of the circle, see Figure \ref{angles-circle}). Thus, since the proposition is only in terms of projective quantities, it remains true for a general conic by applying a projective transformation.
\hfill $\qed$

%

\vspace{0.4cm}
The converse direction is also true, which explains the ubiquitous appearance of conics in geo\-metry:
\begin{prop}\label{Chasles-Steiner-prop}
Let $\alpha$ be a projective transformation between two pencils of lines with base points $P$ and $Q$ respectively. Then the intersection of a line $d$ through $P$ with $\alpha(d)$ describes a conic (when $d$ varies) passing through $P$ and $Q$.
\end{prop}

\begin{figure}[h]
\centering
\includegraphics[height=3.1cm]{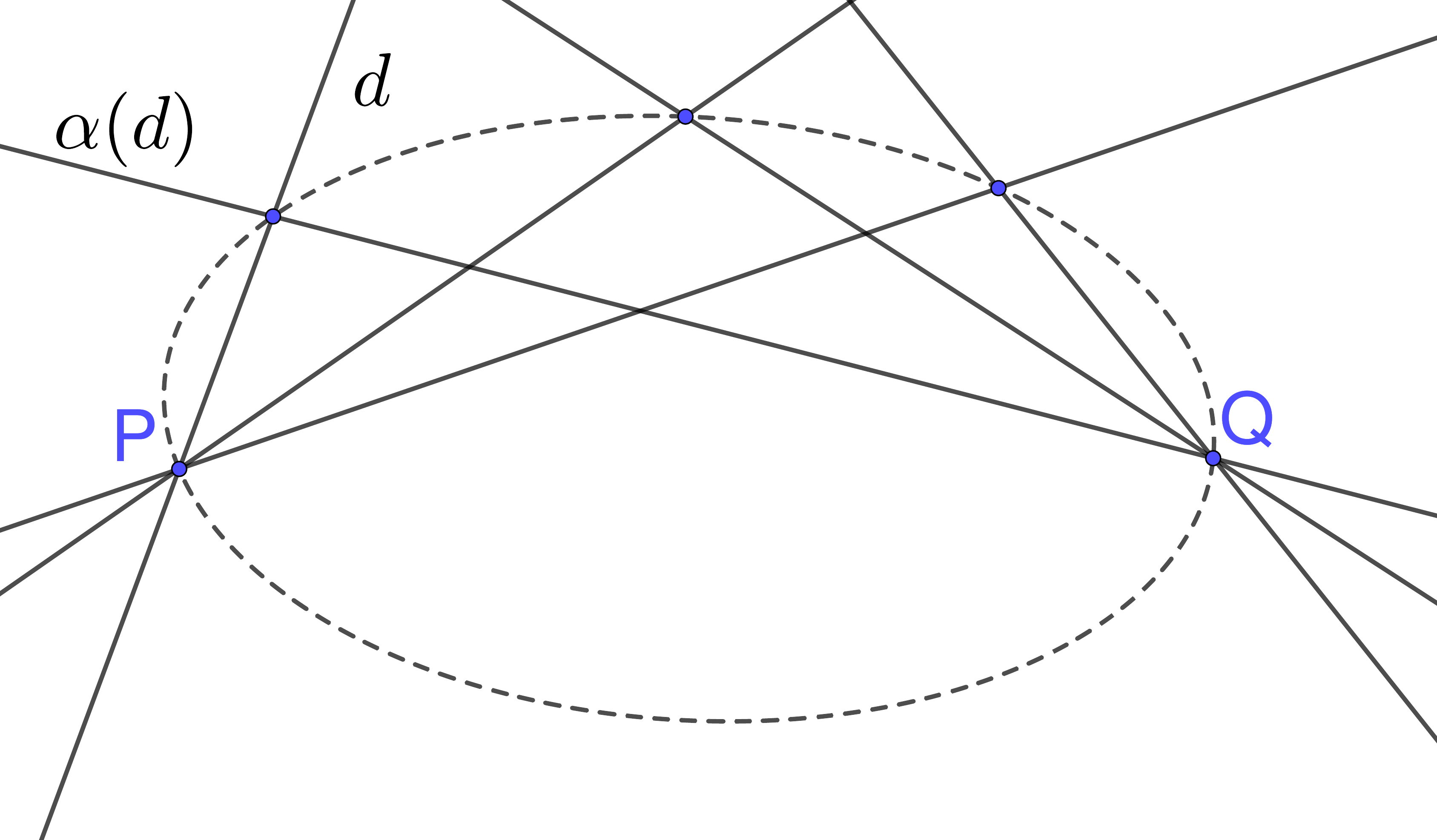}

\vspace{-0.1cm}
\caption{Chasles-Steiner theorem}
\label{Chasles-Steiner}
\end{figure}



The proof is simply the fact that a projective transformation between two pencils of lines is uniquely determined by the image of three lines. This gives three intersection points, and together with $P$ and $Q$ this gives a unique conic, so we can conclude with the previous Proposition \ref{cross-ratio-conic}.

Both Proposition \ref{cross-ratio-conic} and its converse \ref{Chasles-Steiner-prop} are known as the theorem of Chasles-Steiner (see Theorem 5.1 and 5.2 in \cite{Sidler}).

\subsection{Polars}

For a point $P$ in the plane and a given conic $\mc{C}$, we define the \textbf{polar} of $P$ with respect to $\mc{C}$ to be the line defined by all points $Q$ such that denoting by $X$ and $Y$ the intersection of $PQ$ with $\mc{C}$ we have $(P,Q,X,Y)=-1$ (see Figure \ref{polar-def}). The point $P$ is called the \textbf{pole} of its polar.

\begin{figure}[h]
\centering
\includegraphics[height=3cm]{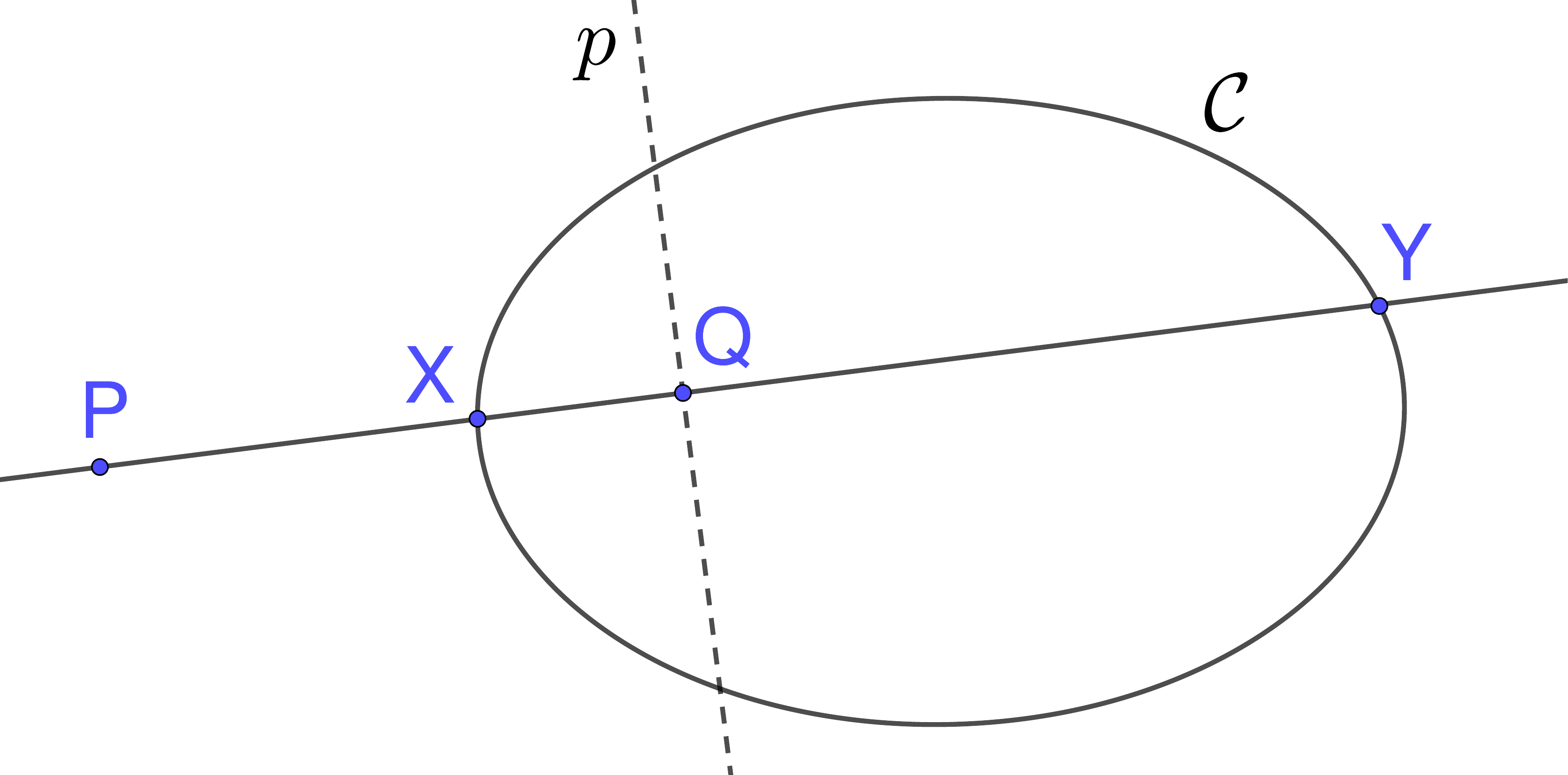}

\caption{Definition of the polar}
\label{polar-def}
\end{figure}

Geometrically, there is the following simple construction of the polar: take any two lines through $P$, each intersecting the conic in two points. Denote these points by $A, B, C, D$, see Figure \ref{polar}. Let $R$ be the intersection of $AC$ and $BD$ and $S$ be the intersection of $AD$ and $BC$. Then the polar is given by $RS$. In particular if $P$ is outside the conic, then the polar is given by the line through the two tangent points.

\begin{figure}[h]
\centering
\includegraphics[height=3.6cm]{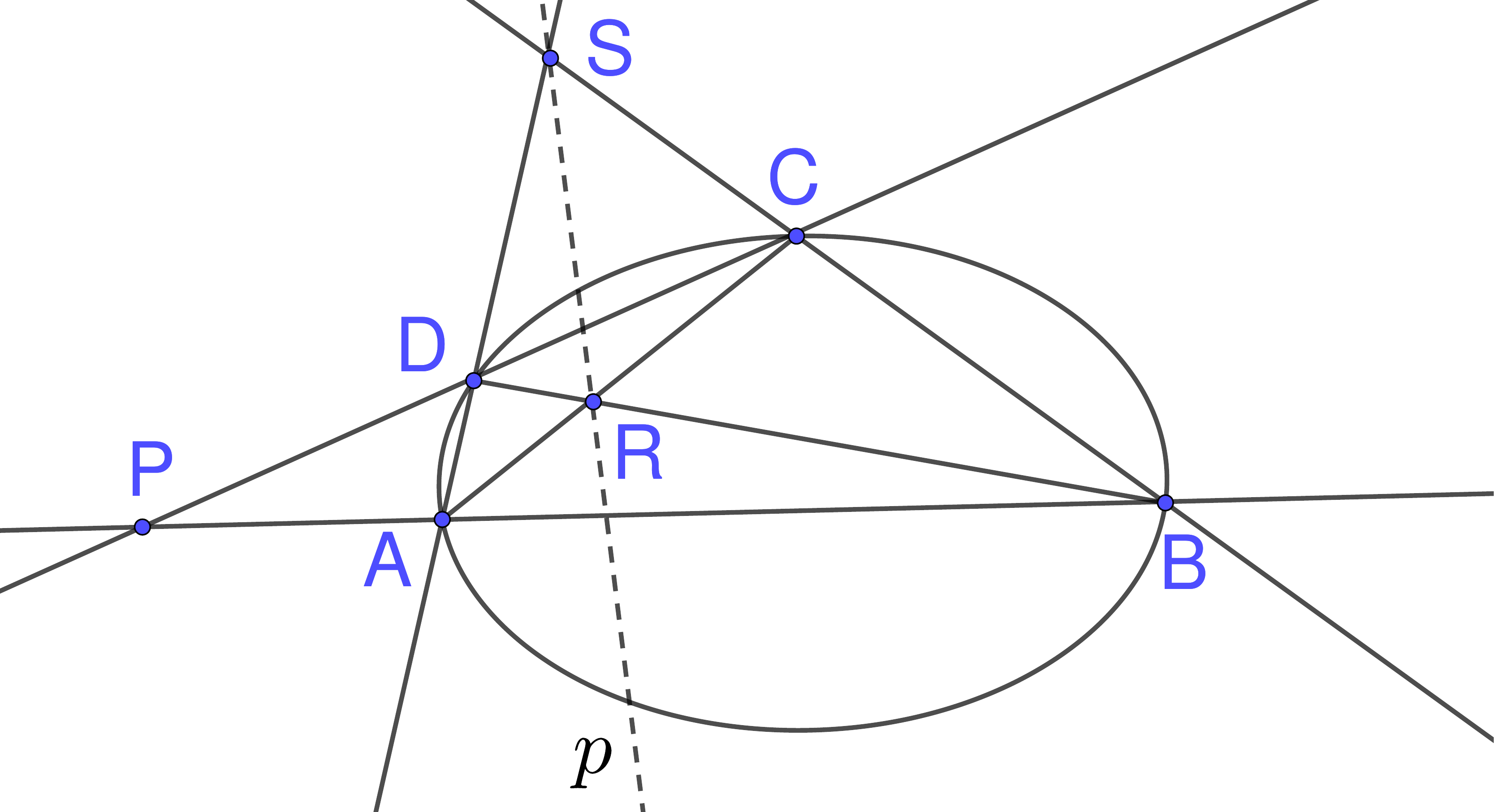}
\hspace{0.5cm}
\includegraphics[height=3.6cm]{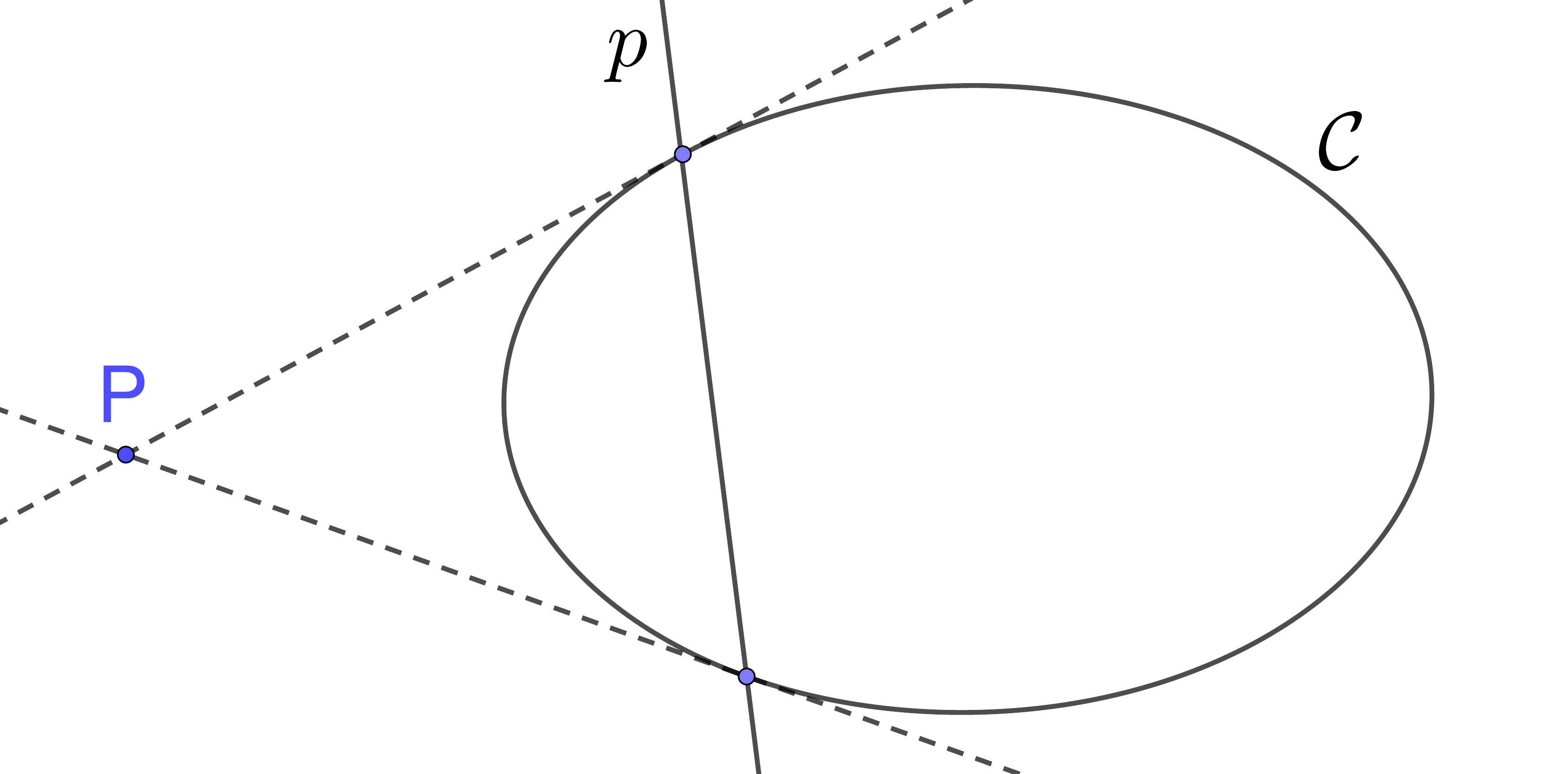}

\caption{Construction of the polar}
\label{polar}
\end{figure}


There is a duality between poles and polars in the following sense: a point $B$ lies on the polar of $A$ iff $A$ lies on the polar of $B$. This comes directly from the definition.

The duality between poles and polars is projective:

\begin{prop}\label{polar-duality}
Given four points $A, B, C, D$ on a line and a conic $\mc{C}$, then the four polars $p_A, p_B, p_C, p_D$ are concurrent and $(A,B,C,D)=(p_A,p_B,p_C,p_D)$.
\end{prop}

\begin{figure}[h]
\centering
\includegraphics[height=3.6cm]{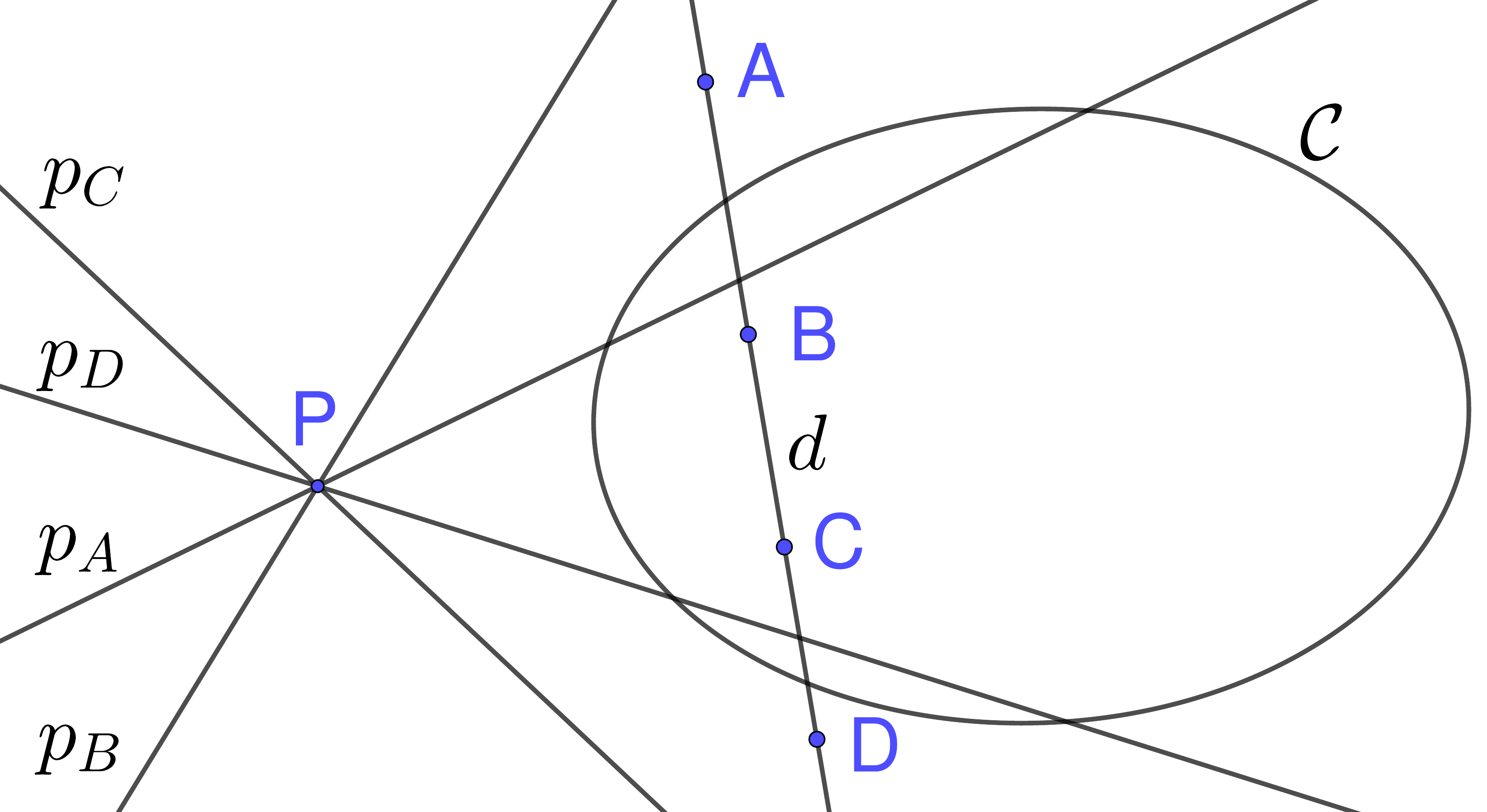}

\caption{Pole-polar duality}
\label{pole-polar}
\end{figure}

\begin{wrapfigure}{r}{0.46\textwidth}
   \includegraphics[width=0.43\textwidth]{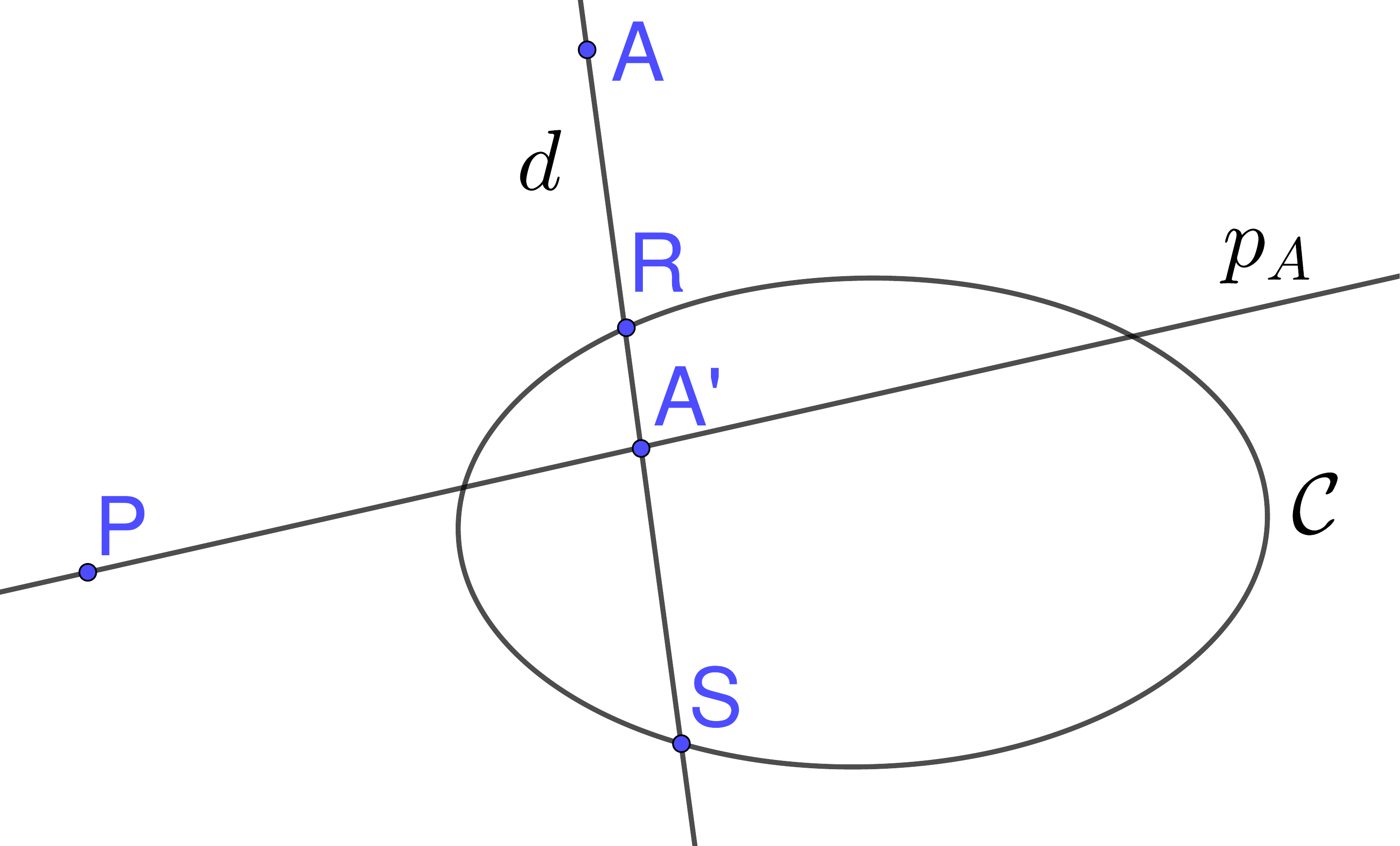}
 
\vspace{0.5cm}
\caption{Pole-polar duality is projective}
\label{polar-proof}
\end{wrapfigure}

\vspace{0.3cm}
\textit{Proof.}
Denote by $d$ the line of the four points $A, B, C, D$ and by $P$ the pole of $d$ (see Figure \ref{polar-proof}). All polars $p_A, p_B, p_C$ and $p_D$ go through $P$, so they are concurrent.

Denote by $R$ and $S$ the intersection of $d$ with the conic $\mc{C}$ (there is always such an intersection if we work in $\C P^2$) and by $A'$ the intersection of $p_A$ and $d$. By definition we have $(A,A',R,S) = -1$, so the map $A \mapsto A'$ is projective (it is the involution on $d$ with fixed points $R$ and $S$).
Finally, the map $A' \mapsto PA'$ is projective (by proposition \ref{cross-ratio-line}). Since $PA'$ is nothing but $p_A$, we see by composition that $A \mapsto p_A$ is projective.
\hfill $\qed$
\vspace{0.5cm}



The dual viewpoint gives the notion of a cross ratio of four tangent lines to a given conic $\mc{C}$ (dual concept of $(A,B,C,D)_{\mc{C}}$). Both notions are linked as follows:

\begin{prop}
Consider four points $A, B, C, D$ on a conic $\mc{C}$ and denote by $t_A, t_B, t_C, t_D$ the four tangents to $\mc{C}$ at these points. Then $$(A,B,C,D)_{\mc{C}}=(t_A,t_B,t_C,t_D)_{\mc{C}}.$$
\end{prop}

\begin{figure}[h]
\centering
\includegraphics[height=4cm]{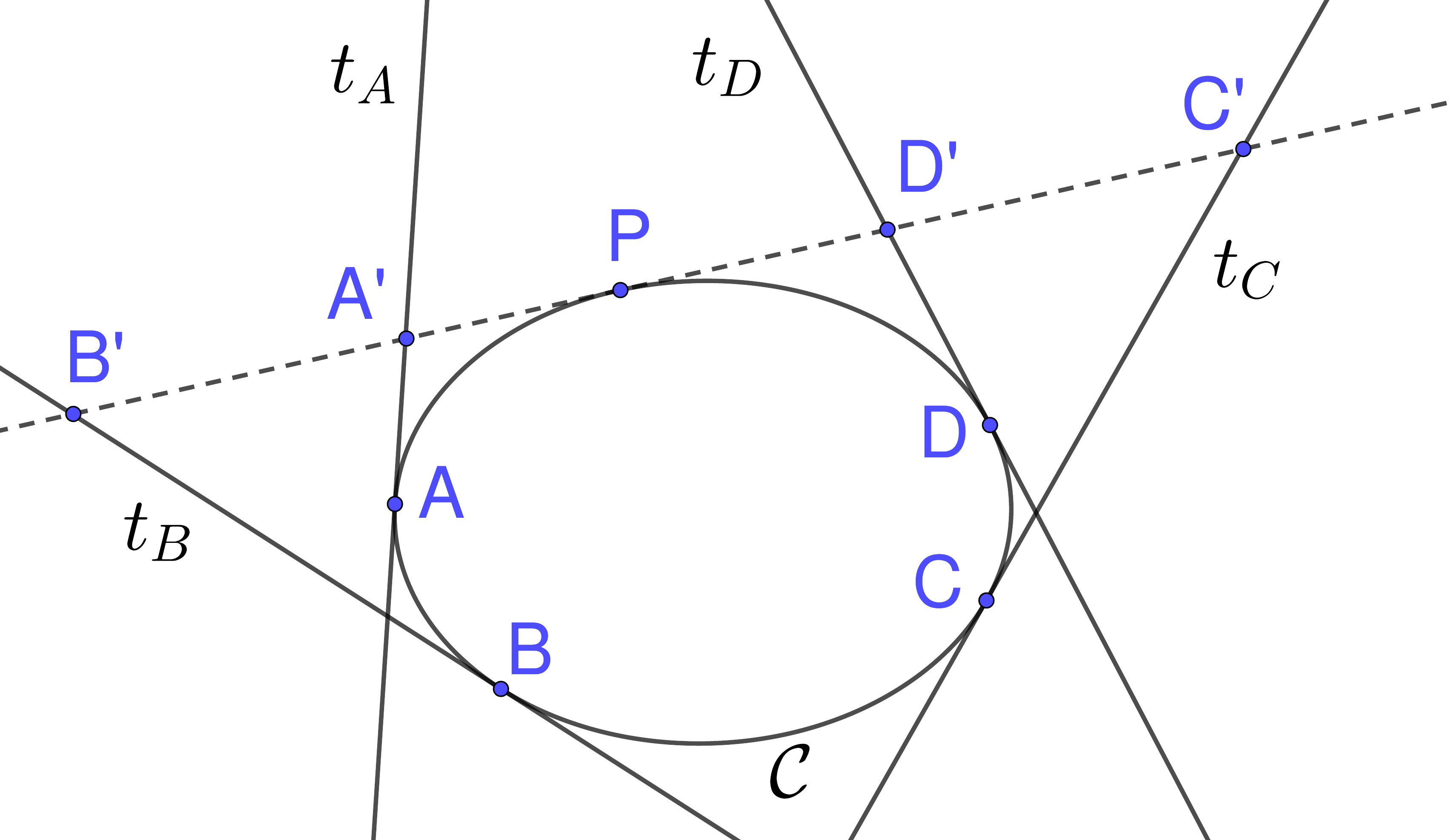}

\caption{Tangents to a conic}
\label{conic-tangents}
\end{figure}

\begin{proof}
Take a fifth point $P$ on $\mc{C}$, see Figure \ref{conic-tangents}. Denote by $A'$ the intersection of $t_A$ with $t_P$ and analogously for $B', C'$ and $D'$. By definition we have $(t_A,t_B,t_C,t_D)_{\mc{C}} = (A',B',C',D')$. Thus we get

\begin{align*}
(t_A,t_B,t_C,t_D)_{\mc{C}} &= (A',B',C',D') & \text{ by definition}\\
&= (PA,PB,PC,PD) & \text{ by Proposition \ref{polar-duality}} \\
&= (A,B,C,D)_{\mc{C}} & \text{ by definition}
\end{align*}
\end{proof}

Using Proposition \ref{cross-ratio-line}, the cross ratio on a conic can be projected onto a line:
\begin{prop}
Let $A, B, C, D$ and $P$ be points on a common conic $\mc{C}$ and $d$ be a line (see Figure \ref{conic-to-line}). The intersection of $PA$ with $d$ is denoted by $A'$ and analogously are defined $B', C'$ and $D'$. Then 
$$(A,B,C,D)_{\mc{C}} = (A',B',C',D').$$
\end{prop}

\begin{figure}[h]
\centering
\includegraphics[height=4cm]{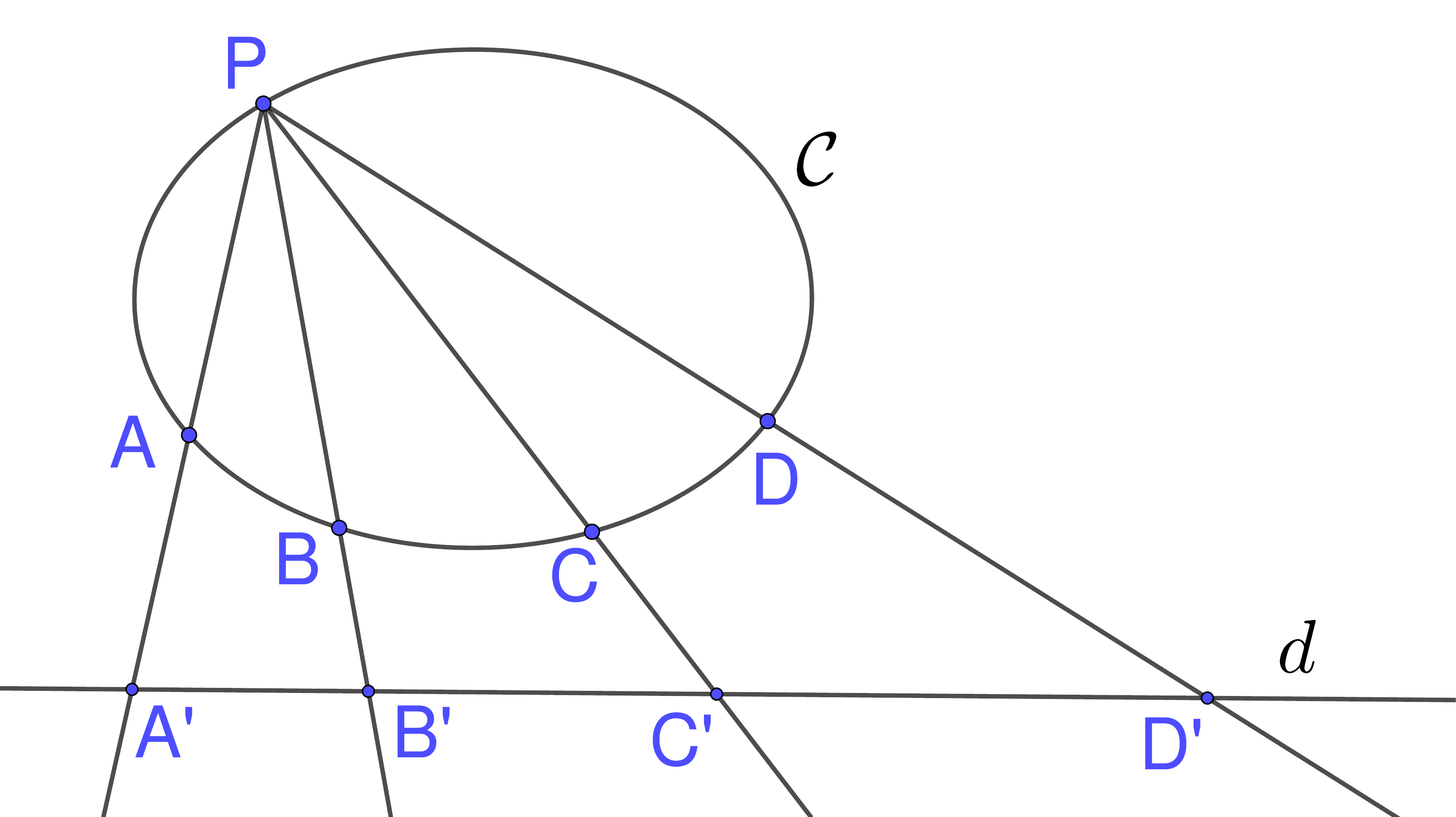}

\caption{Conic to line}
\label{conic-to-line}
\end{figure}

\medskip 
Every point in the plane determines an involution on a conic:
\begin{prop}\label{conic-to-conic-prop}
Let $P$ be a point in the plane. For a point $A$ on a conic $\mc{C}$, denote the second intersection of $PA$ with $\mc{C}$ by $A'$ (see Figure\ref{conic-to-conic}). Then the map $A\mapsto A'$ is projective, i.e. given four points $A, B, C, D$ on $\mc{C}$, we have
$$(A,B,C,D)_{\mc{C}}=(A',B',C',D')_{\mc{C}}.$$
\end{prop}

\begin{figure}[h]
\centering
\includegraphics[height=4cm]{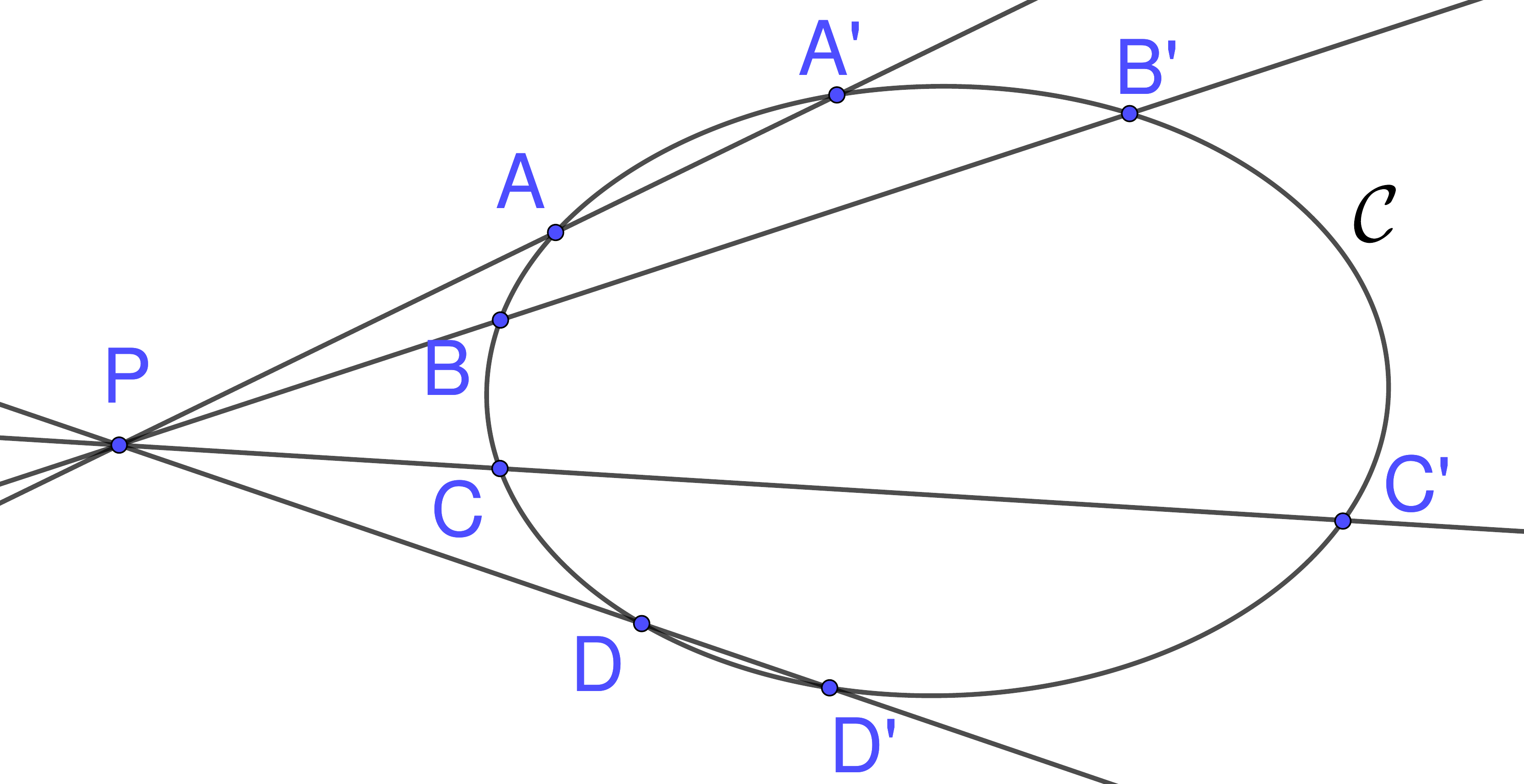}

\caption{Conic to conic}
\label{conic-to-conic}
\end{figure}

Notice that the map $A \mapsto A'$ is involutive. A theorem due to Frégier asserts that any involution of a conic is induced by some point $P$ in the plane (see Theorem 5.4 in \cite{Sidler}).

\begin{proof}
We only indicate the idea of the proof: consider the involution of the plane with center $P$, axis $p_P$ (the polar of $P$ with respect to $\mc{C}$) and which takes $A$ to $A'$. The involution is uniquely determined by this. One easily shows that this involution preserves the conic $\mc{C}$ and induces the involution described in the proposition.
\end{proof}

\subsection{Pencil of conics}

Finally we need some preliminaries on pencils of conics.
Given two conics defined by equations $E_1=0$ and $E_2=0$ (quadratic polynomials in two variables), the pencil $\mc{F}$ generated by them are all conics with equation $\l E_1 + \mu E_2 = 0$.
Geometrically two conics intersect generically in four points, called the \textbf{base points} of the pencil. A conic belongs to their pencil $\mc{F}$ iff it goes through the four base points.
Every pencil of conics contains three degenerate conics, consisting of pairs of lines through the base points. The double points of these degenerate conics are called \textbf{double points} of the pencil. In Figure \ref{pencil-conic} you see a pencil of conics with base points $A, B, C, D$ and double points $R, S$ and $T$.

\begin{figure}[h]
\centering
\includegraphics[height=4cm]{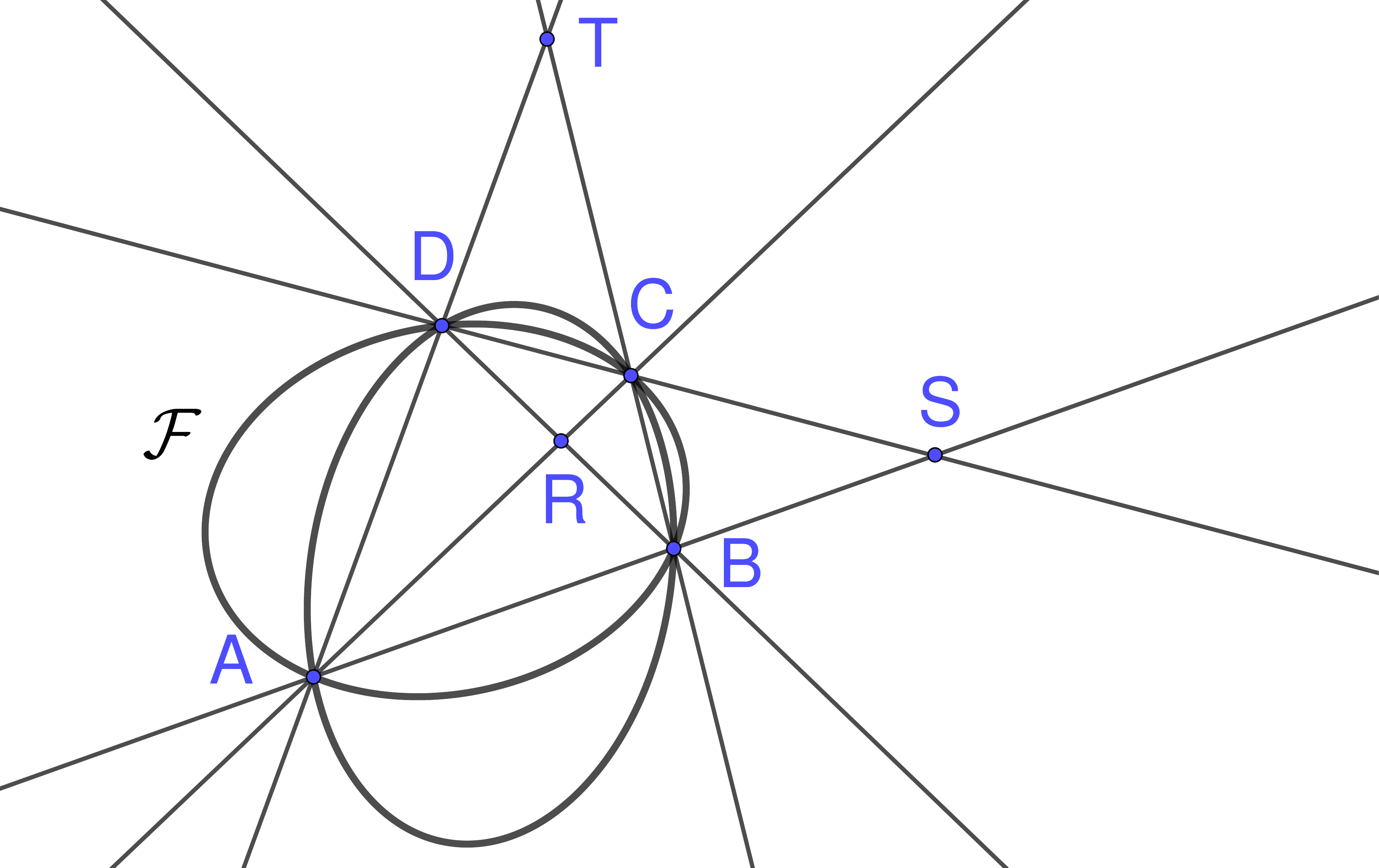}

\caption{Pencil of conics}
\label{pencil-conic}
\end{figure}

A pencil of conics $\mc{F}$ induces a projective involution on every straight line, due to a famous theorem of Desargues:
\begin{prop}[third theorem of Desargues]\label{desargues-prop}
Let $d$ be a line not passing through any base point of $\mc{F}$. For $P \in d$ there is a unique conic $\mc{C} \in \mc{F}$ going through $P$. Denote by $P'$ the second intersection of $\mc{C}$ with $d$. The map $P \mapsto P'$ is a projective involution on $d$.
\end{prop}
This induced map on $d$ is called \textbf{Desargues involution}. The proof can be found in \cite{Sidler}, Theorem 6.1.

\begin{figure}[h]
\centering
\includegraphics[height=4cm]{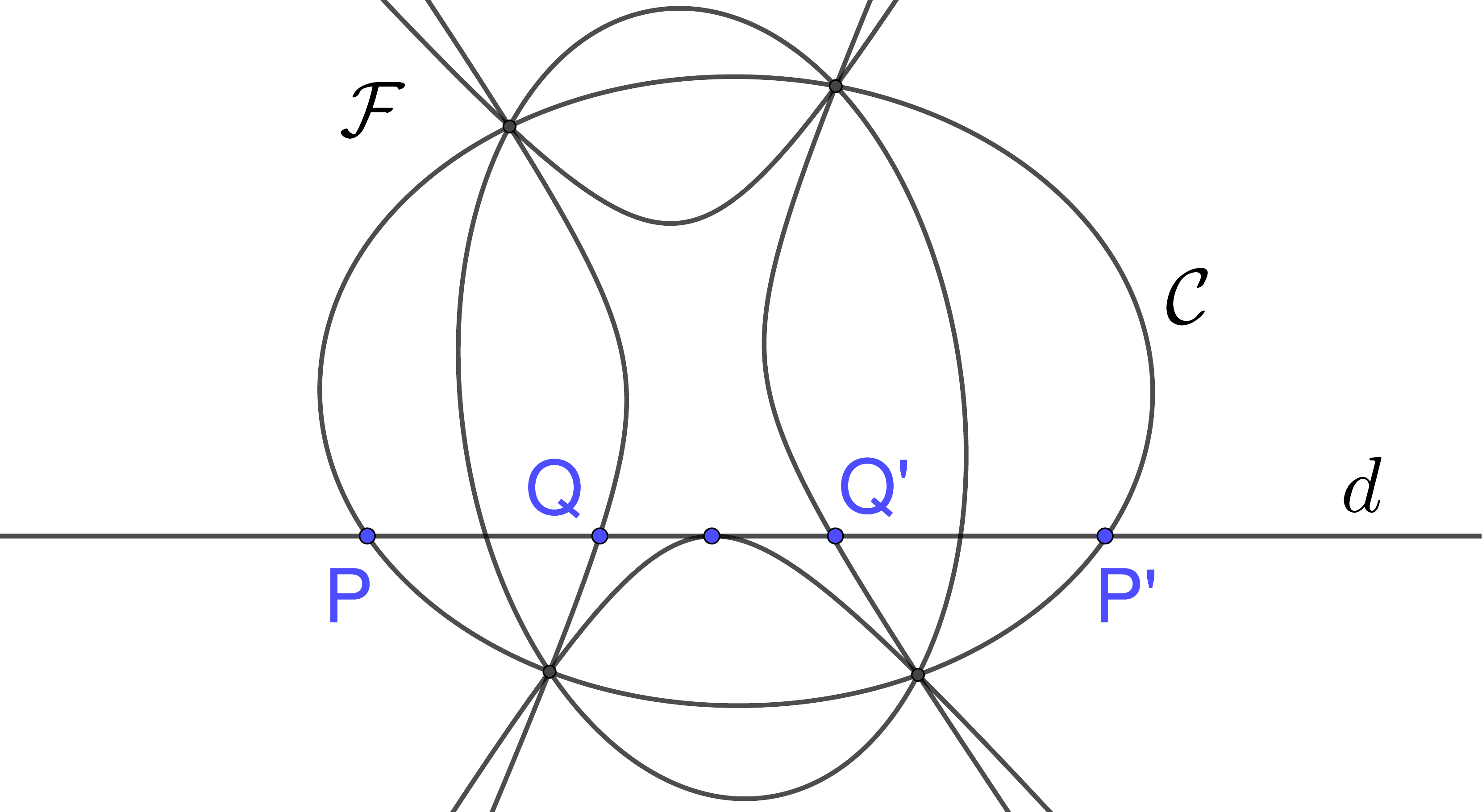}

\caption{Desargues involution}
\label{desargues}
\end{figure}

Further, any pencil of conics $\mc{F}$ induces a transformation of the whole plane:
\begin{prop}\label{conj-point}
Let $P$ be a point in the plane, not a double point of the pencil of conics $\mc{F}$. The set of all polars of $P$ with respect to conics in $\mc{F}$ are concurrent to a single point $P'$.
\end{prop}
We call $P'$ the \textbf{conjugated point} of $P$ with respect to $\mc{F}$.

\begin{figure}[h]
\centering
\includegraphics[height=4cm]{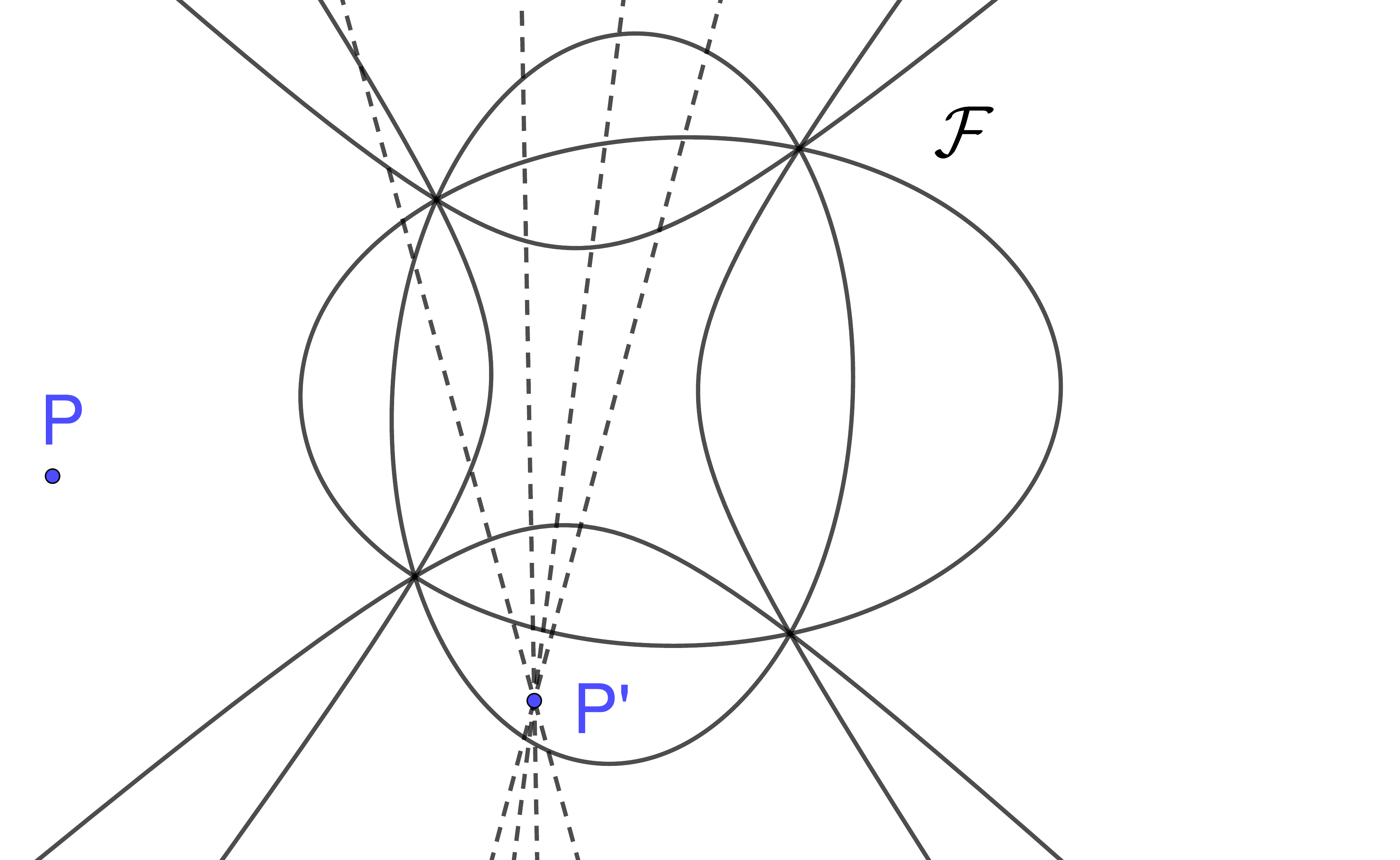}

\caption{Conjugated point}
\label{point-conjugue}
\end{figure}

\begin{proof}
Denote by $P'$ the intersection of the polars of $P$ with respect to two conics $\Gamma_1$ and $\Gamma_2$ of the pencil $\mc{F}$. Denote by $M_i$ and $N_i$ the intersection points of $PP'$ with $\Gamma_i$. The Desargues involution exchanges $M_i$ and $N_i$ and is uniquely determined by this. Furthermore, since $(P,P',M_i,N_i)=-1$ (by definition of the polar), we see that the fixed points of the Desargues involution are $P$ and $P'$ (this means that the conic in $\mc{F}$ passing through $P$ is tangent to $PP'$).
For any conic $\Gamma \in \mc{F}$, denote by $M, N$ the intersection of $PP'$ with $\Gamma$. Then $(P,P',M,N)=-1$ since $M$ and $N$ are exchanged by the Desargues involution. Therefore the polar of $P$ with respect to $\Gamma$ passes through $P'$.
\end{proof}

\begin{figure}[h]
\centering
\includegraphics[height=4cm]{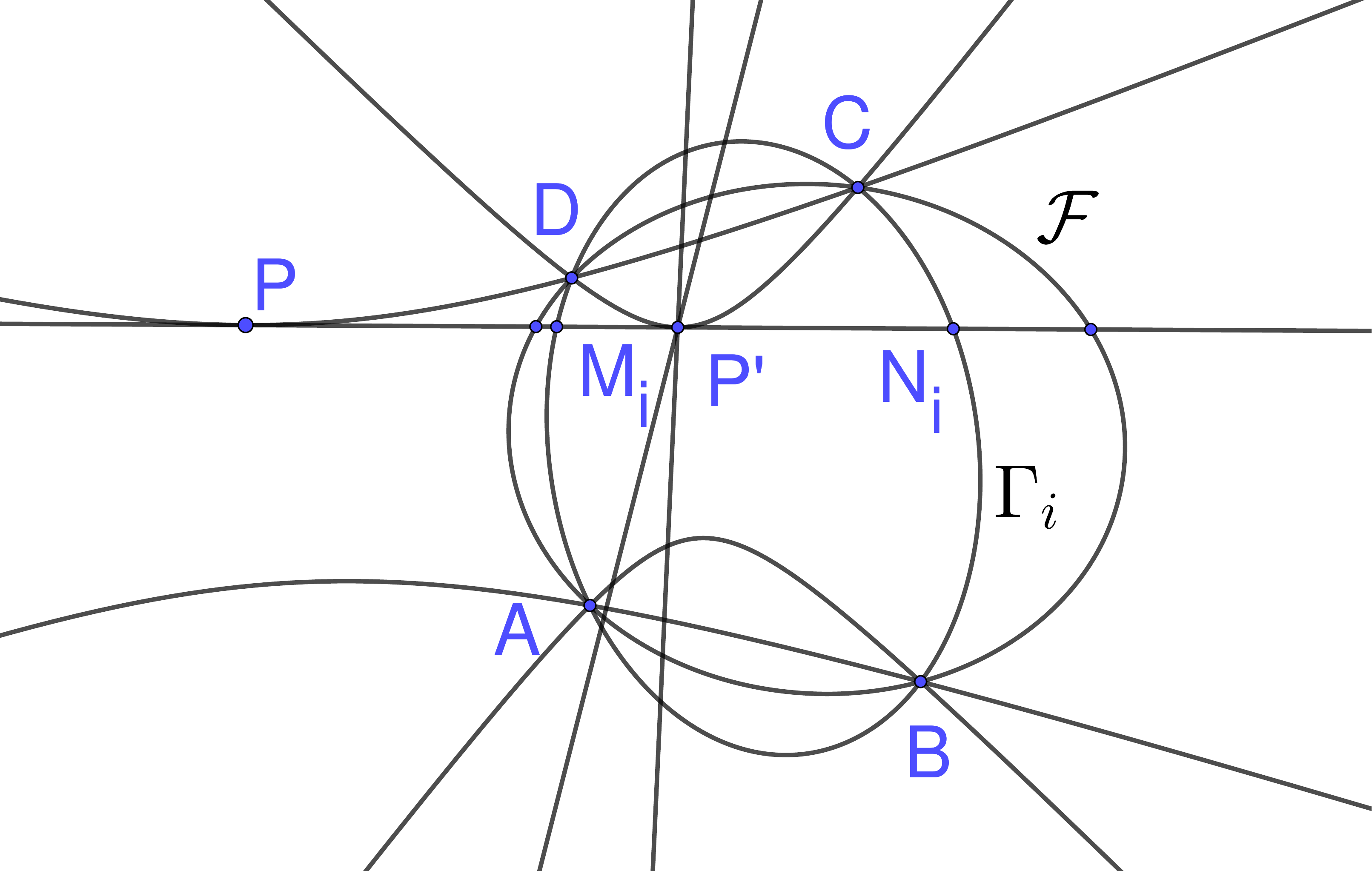}

\caption{Proof of Proposition \ref{conj-point}}
\label{point-conjugue-proof}
\end{figure}

\begin{Remark}
There is also a simple algebraic proof: the equation of the polar of $P$ with respect to a conic given by an equation which is a linear combination of polynomials $E_1$ and $E_2$ is itself a linear combination of the equations of the polars with respect to $E_1$ and $E_2$. Thus, the set of all polars forms a pencil.
See \cite{Berger}, 16.5.3, for more details.
\end{Remark}

The conjugation with respect to a pencil of conics is not projective. It is a \textit{quadratic birational} transformation. The image of a line is given by the following proposition:
\begin{prop}\label{conj-line}
The point-wise image of a straight line $d$ under conjugation with respect to a pencil of conics $\mc{F}$ is a conic which passes through the double points of $\mc{F}$.
The same result is obtained by taking the set of all poles of $d$ with respect to $\mc{F}$.
\end{prop}

\begin{figure}[h]
\centering
\includegraphics[height=4cm]{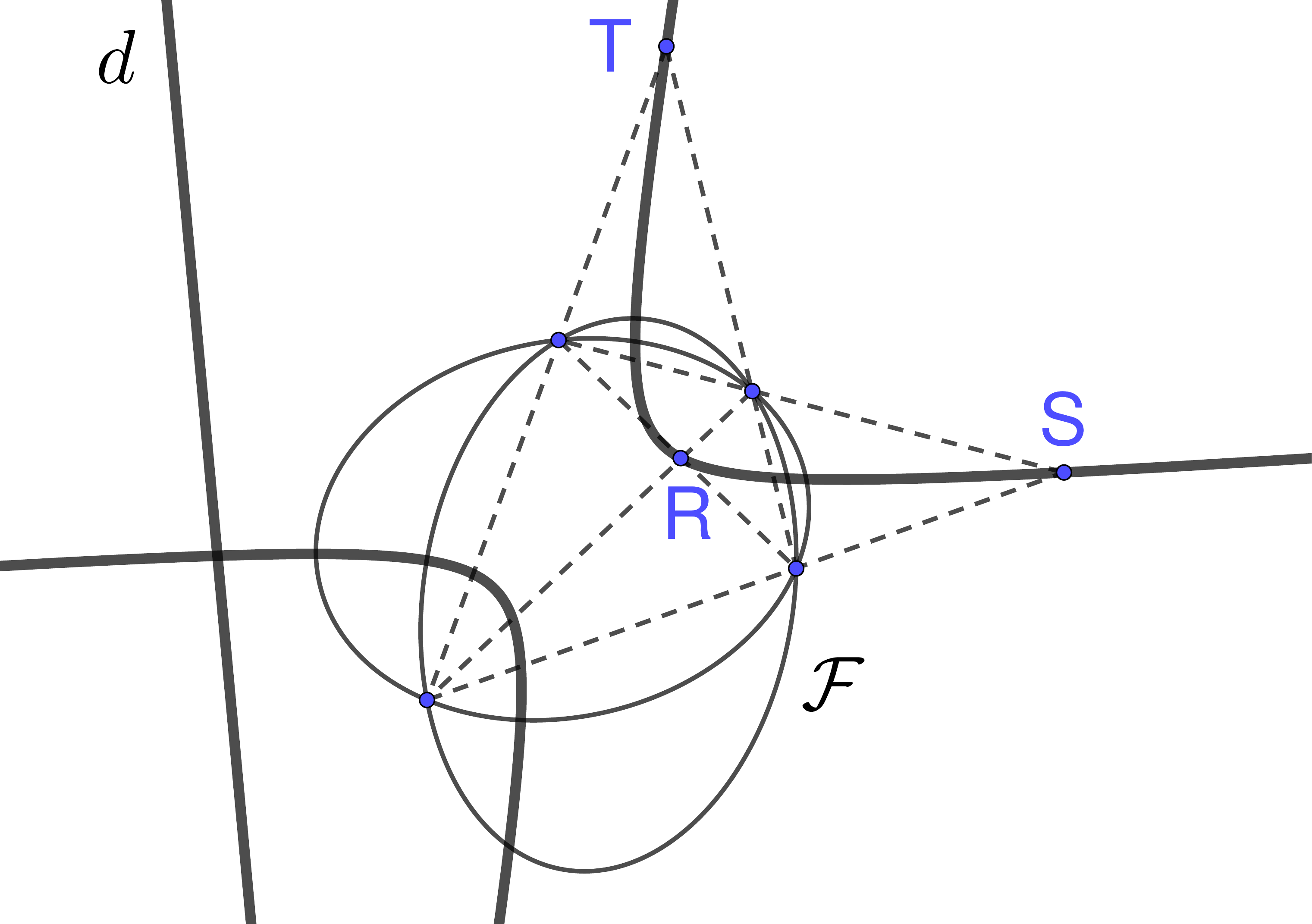}
\hspace{0.5cm}
\includegraphics[height=4cm]{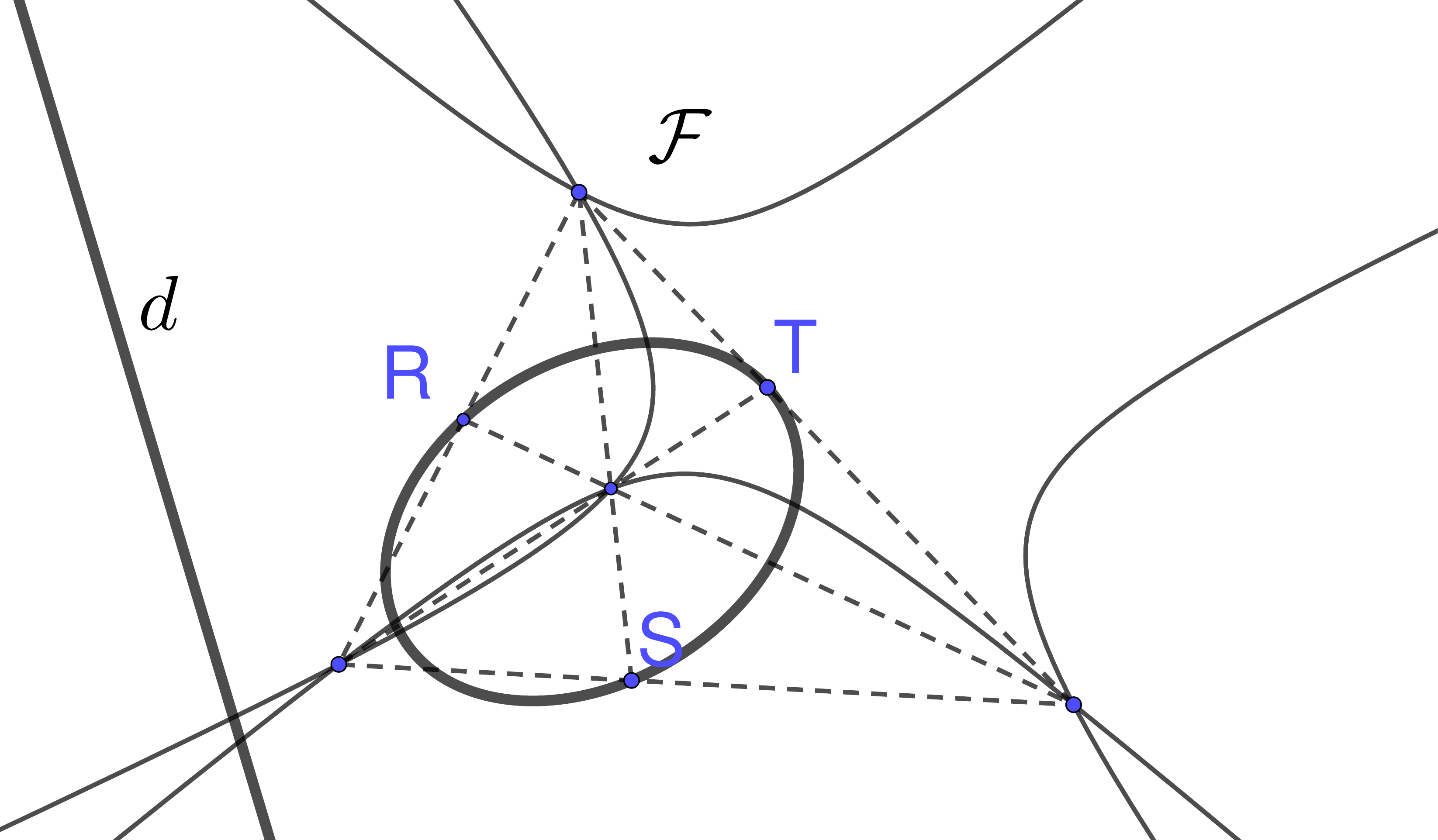}

\caption{Conjugated line gives a conic (two examples)}
\label{droite-conjugue}
\end{figure}

\begin{proof}
Let $P$ be a point on $d$. Denote by $R, S, T$ the double points of the pencil of conics. The line $RP'$ is the polar of $P$ with respect to the degenerate conic defining $R$. Hence by Proposition \ref{polar-duality} the map $P\mapsto RP'$ is projective. In the same vain, the map $P\mapsto SP'$ is projective. Thus the map $RP' \mapsto SP'$ is projective.

By the theorem of Chasles-Steiner, Proposition \ref{Chasles-Steiner}, the image of $P'$ is a conic passing through $R$ and $S$. Repeating the argument with $S$ and $T$ gives that this conic also passes through $T$.

To prove that this conic is the set of all polars of $d$ with respect to the pencil $\mc{F}$, take again a point $P$ on $d$. We show that there is a conic in $\mc{F}$ such that the pole of $d$ is $P'$. Since $P''=P$ all polars of $P'$ with respect to $\mc{F}$ pass through $P$. We will see in Theorem \ref{thm1} that the map $\Gamma \in \mc{F} \mapsto p_{P'}$ is projective, so in particular surjective. Hence there is a conic whose polar is $d$ which concludes the proof.
\end{proof}

Even if the conjugation with respect to a pencil of conics is not projective, it does preserve the cross ratio in some sense:

\begin{prop}\label{conjugation-pencil}
Let $A, B, C, D$ be four points on a common line $d$ and $\mc{F}$ be a pencil of conics. Denote by $A', B', C', D'$ the conjugated points and by $\mc{C}$ the conic which is the image of $d$ under conjugation. Then $$(A,B,C,D)=(A',B',C',D')_{\mc{C}}.$$
\end{prop}

\begin{figure}[h]
\centering
\includegraphics[height=4cm]{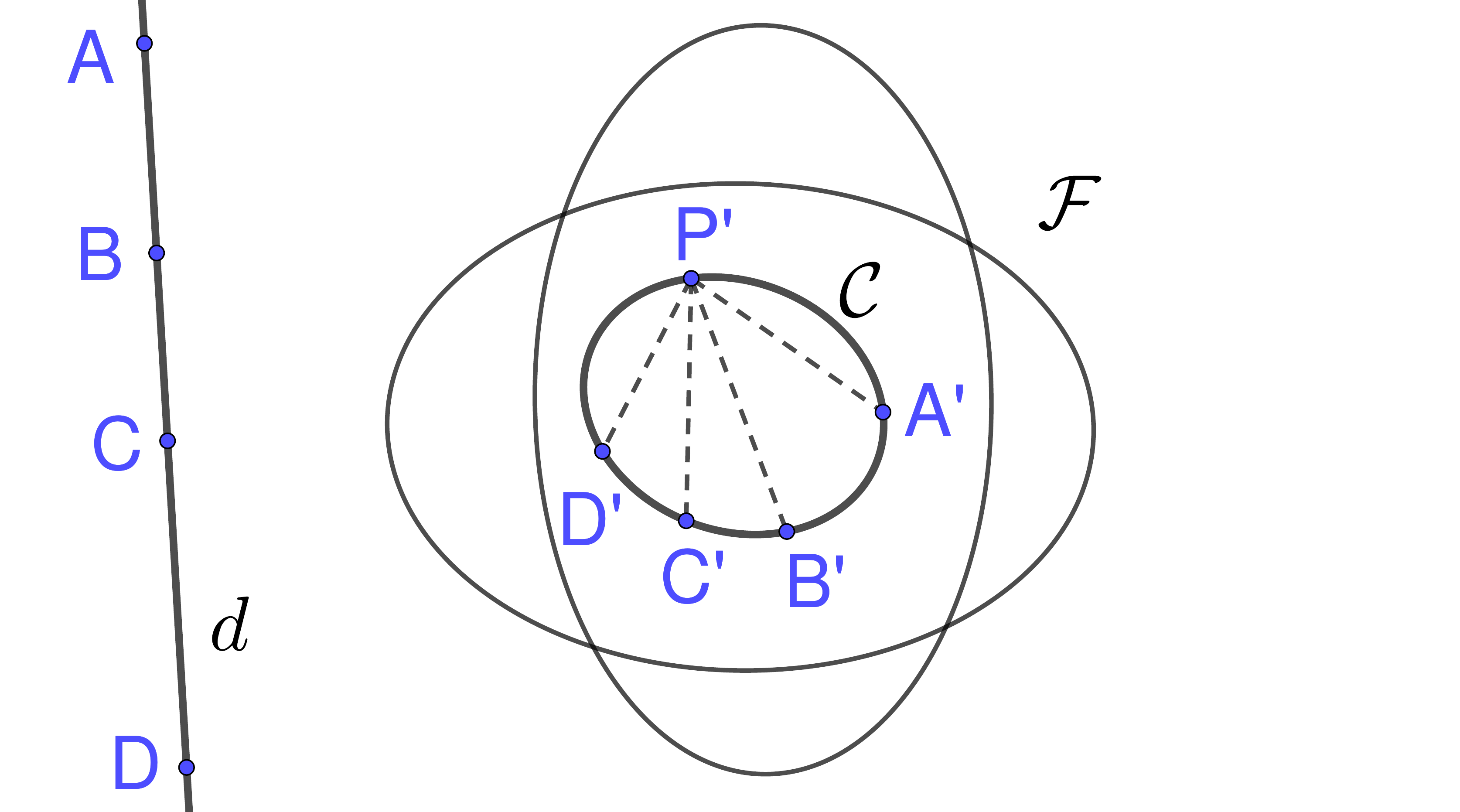}

\caption{Cross ratio preservation under conjugation (figure not realistic)}
\label{conjugation-cross-ratio}
\end{figure}

\begin{proof}
Let $\Gamma$ be a conic in $\mc{F}$ such that the pole of $d$ with respect to $\Gamma$ is a point $P'$ different from $A', B', C'$ and $D'$ (see Figure \ref{conjugation-cross-ratio}). The polar of $A$ with respect to $\Gamma$ is then $P'A'$ (every polar to a conic in $\mc{F}$ passes through $A'$) and the same holds true for the other points $B, C, D$.

Thus
\begin{align*}
(A,B,C,D) &= (P'A',P'B',P'C',P'D') & \text{ by Proposition \ref{polar-duality}} \\
&= (A',B',C',D')_{\mc{C}}.
\end{align*}
\end{proof}

\section{Cross ratio in a pencil of conics}\label{section2}

In this section, we explore geometric ways to find the cross ratio of four conics in a common pencil.
Apart from the theorems \ref{thm1} and \ref{thm3}, all the others seem to be new. 

For the whole section we fix a pencil of conics $\mc{F}$ with base points $A, B, C, D$, double points $R, S, T$ and we consider four conics $\Gamma_1, \Gamma_2, \Gamma_3$ and $\Gamma_4$ in $\mc{F}$.

Probably the first thing one might try to do to determine the cross ratio of the four conics $\Gamma_i$ is the following:
\begin{thm}\label{thm1}
Let $P$ be a point in the plane, different from $R, S, T$, and denote by $p_i$ the polar of $P$ with respect to $\Gamma_i$ ($i=1, 2, 3 ,4$). Then
$$(\Gamma_1,\Gamma_2,\Gamma_3,\Gamma_4) = (p_1,p_2,p_3,p_4).$$
\end{thm}
Notice that from Proposition \ref{conj-point} we know that the polars $p_i$ are concurrent in $P'$. This theorem is well-known, see for instance \cite{Berger} 16.5.3.3, or \cite{Michel} on page 41.

\begin{figure}[h]
\centering
\includegraphics[height=4cm]{point-conjugue.pdf}
\hspace{0.5cm}
\includegraphics[height=4cm]{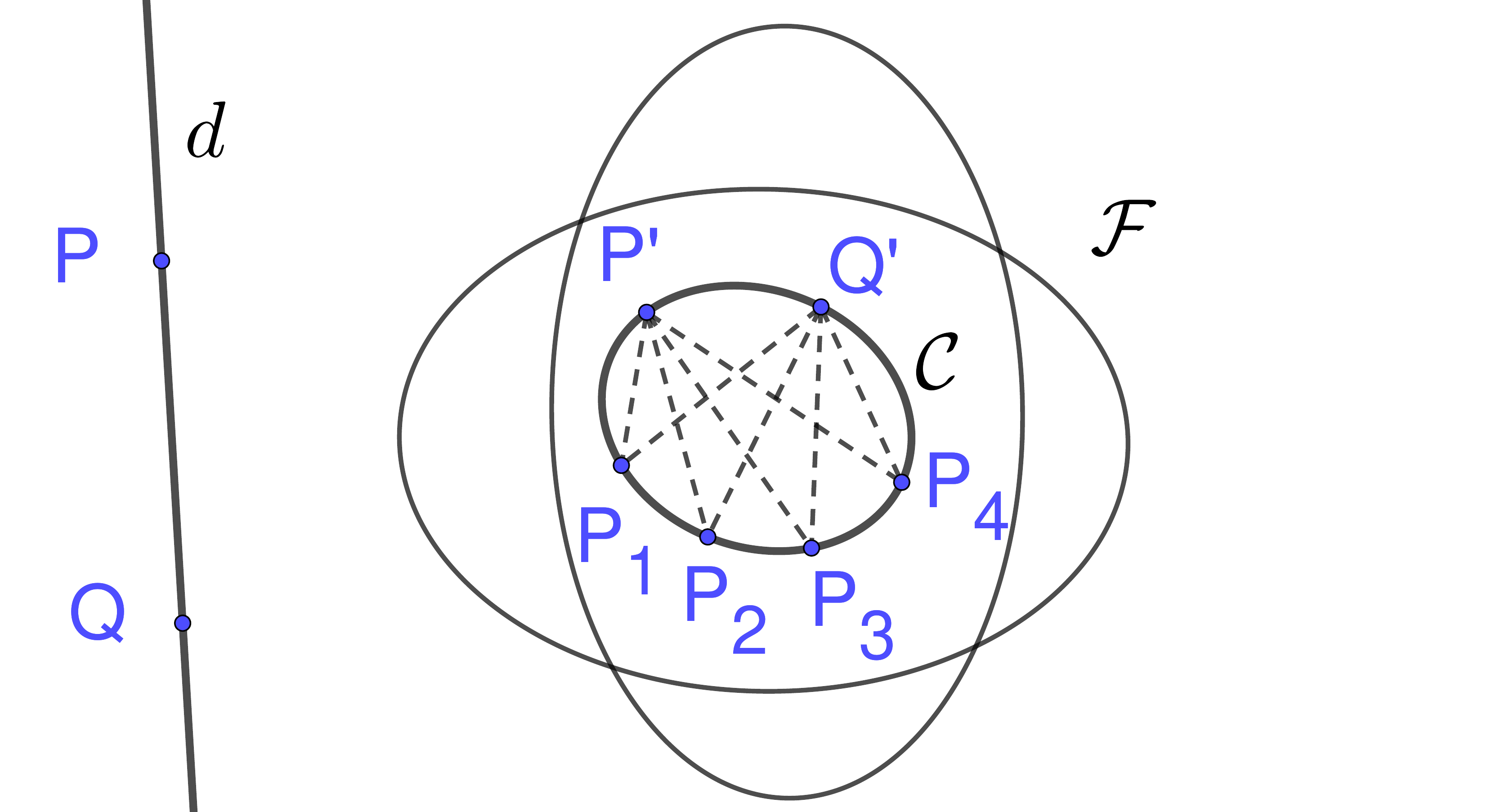}

\caption{Theorem \ref{thm1} and its proof (figure not realistic)}
\label{thm1-image}
\end{figure}

\begin{proof}
We prove that $(p_1,p_2,p_3,p_4)$ is independent of $P$. The fact that it coincides with the double ratio of the four conics will be done using theorem \ref{thm4}.

So take two points $P$ and $Q$. Denote by $d$ the line defined by these two points and by $\mc{C}$ the conjugation of $d$ with respect to the pencil $\mc{F}$ (see Proposition \ref{conj-line}).
The polar of $P$ with respect to $\Gamma_i$ intersects $\mc{C}$ in $P'$ and in a point denoted by $P_i$.
The polar of $Q$ with respect to $\Gamma_i$ is then given by $Q'P_i$, since $P_i$ is the pole of $d$ with respect to $\Gamma_i$.
Finally we get
$$(P'P_1,P'P_2,P'P_3,P'P_4)=(Q'P_1,Q'P_2,Q'P_3,Q'P_4)$$ by Proposition \ref{cross-ratio-conic}.
\end{proof}

When $P$ is one of the four base points of $\mc{F}$, we get:
\begin{coro}
The cross ratio of four conics in a pencil is given by the cross ratio of the four tangents at a base point.
\end{coro}

We can also determine the cross ratio of four conics using a line:

\begin{thm}\label{thm2}
Let $d$ be a straight line. Denote by $\mc{C}$ the conic which is the image of $d$ with respect to $\mc{F}$. Denote further by $P_i$ the pole of $d$ with respect to $\Gamma_i$. Then $$(\Gamma_1,\Gamma_2,\Gamma_3,\Gamma_4)=(P_1,P_2,P_3,P_4)_{\mc{C}}.$$
\end{thm}

\begin{figure}[h]
\centering
\includegraphics[height=4cm]{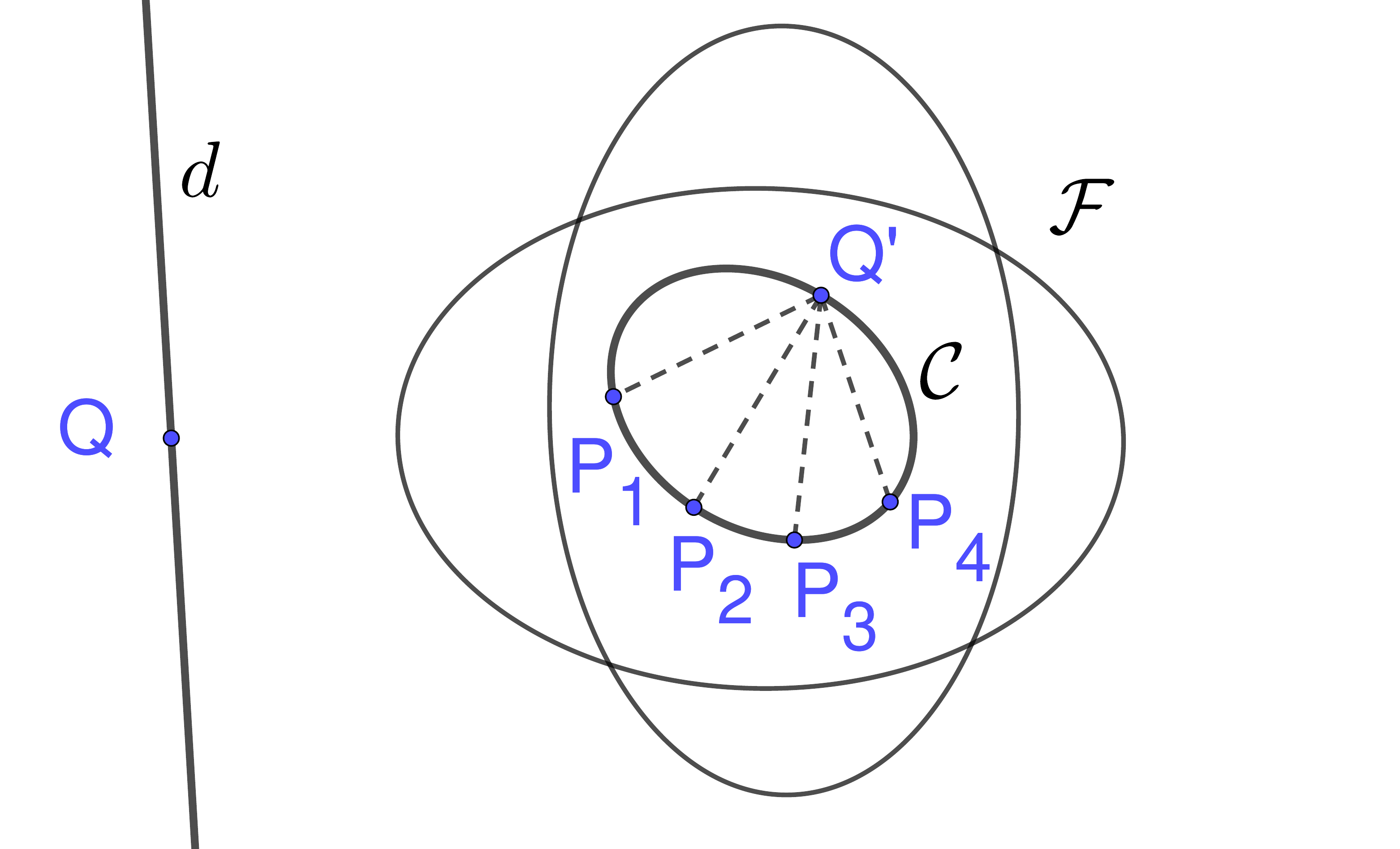}

\caption{Proof of theorem \ref{thm2} (figure not realistic)}
\label{thm2-image}
\end{figure}

\begin{proof}
Let $Q$ be on $d$ such that its conjugated point $Q'$ is different from the $P_i$. The polar of $Q$ with respect to $\Gamma_i$ is then given by $Q'P_i$. Hence:
\begin{align*}
(P_1,P_2,P_3,P_4)_{\mc{C}} &= (Q'P_1,Q'P_2,Q'P_3,Q'P_4) &\\
&= (\Gamma_1,\Gamma_2,\Gamma_3,\Gamma_4) & \text{ by Theorem \ref{thm1}.}
\end{align*}
\end{proof}

When $d$ is the line at infinity, we get:
\begin{coro}
The cross ratio of four conics in a pencil is the cross ratio of the four centers, taken on the conic of centers.
\end{coro}

The geometric simplest way is probably the following:

\begin{thm}\label{thm3}
Consider a line $d$ passing through one of the base points of the pencil $\mc{F}$. Denote by $M_i$ the second intersection of $d$ with $\Gamma_i$. Then $$(\Gamma_1,\Gamma_2,\Gamma_3,\Gamma_4)=(M_1,M_2,M_3,M_4).$$
\end{thm}

We have not found this theorem in the literature, but its dual version is given without proof in \cite{Ingrao}, chapter VIII, 3.29. So we believe that Theorem \ref{thm3} is known.

\begin{figure}[h]
\centering
\includegraphics[height=4.5cm]{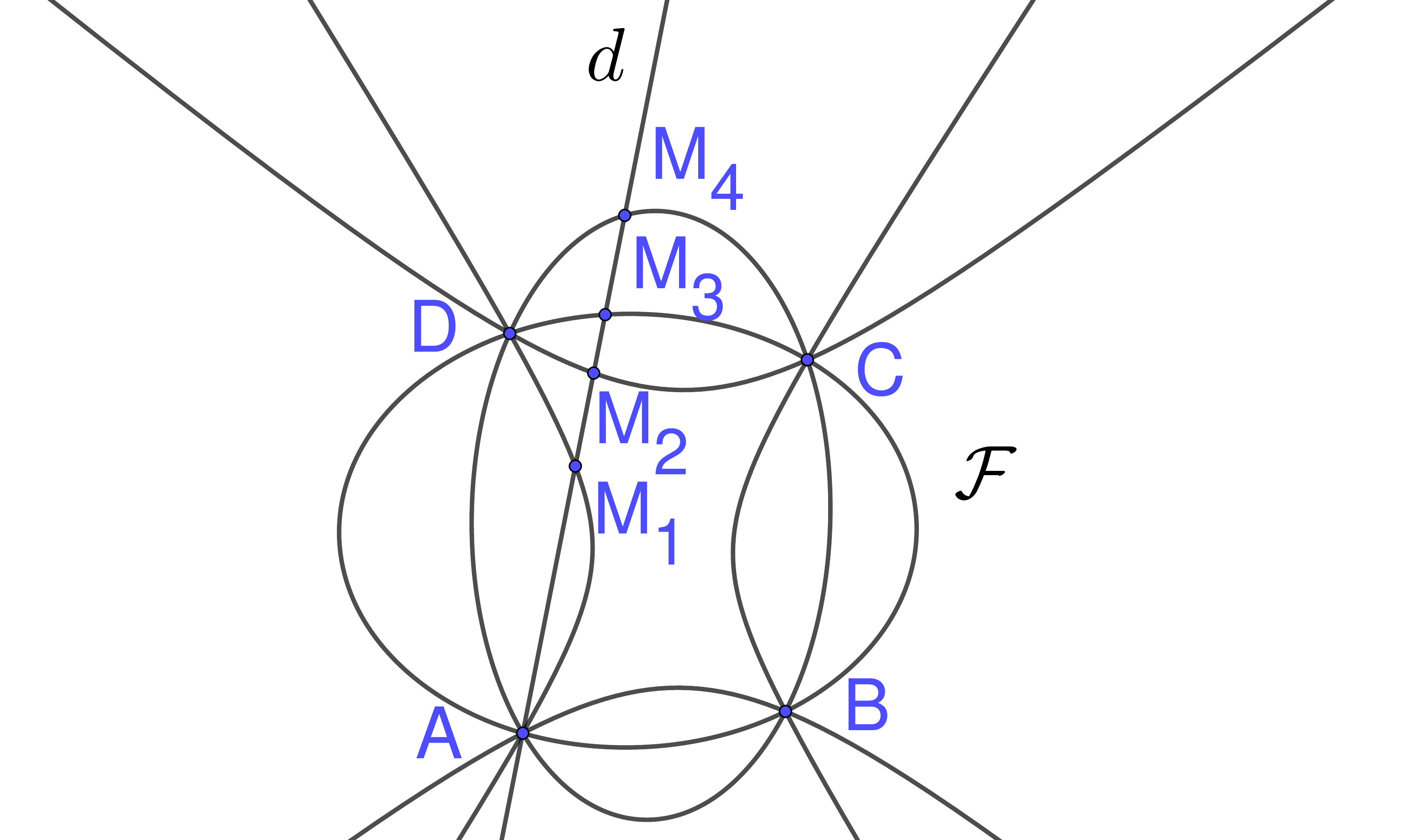}

\caption{Theorem \ref{thm3}}
\label{thm3-image}
\end{figure}

Before proving the theorem, we first need a lemma (Exercise 6.8 in \cite{Sidler}):
\begin{lemma}\label{lemma1}
Let $P$ be a point in the plane. For $\Gamma \in \mc{F}$, denote by $M$ and $N$ the second intersection points of $PA$ and $PB$ with $\Gamma$ (see Figure \ref{lemma-image}). Then the line $MN$ pass through a fixed point $X$ lying on $CD$, independent of $\Gamma \in \mc{F}$.
\end{lemma}

\begin{figure}[h]
\centering
\includegraphics[height=4cm]{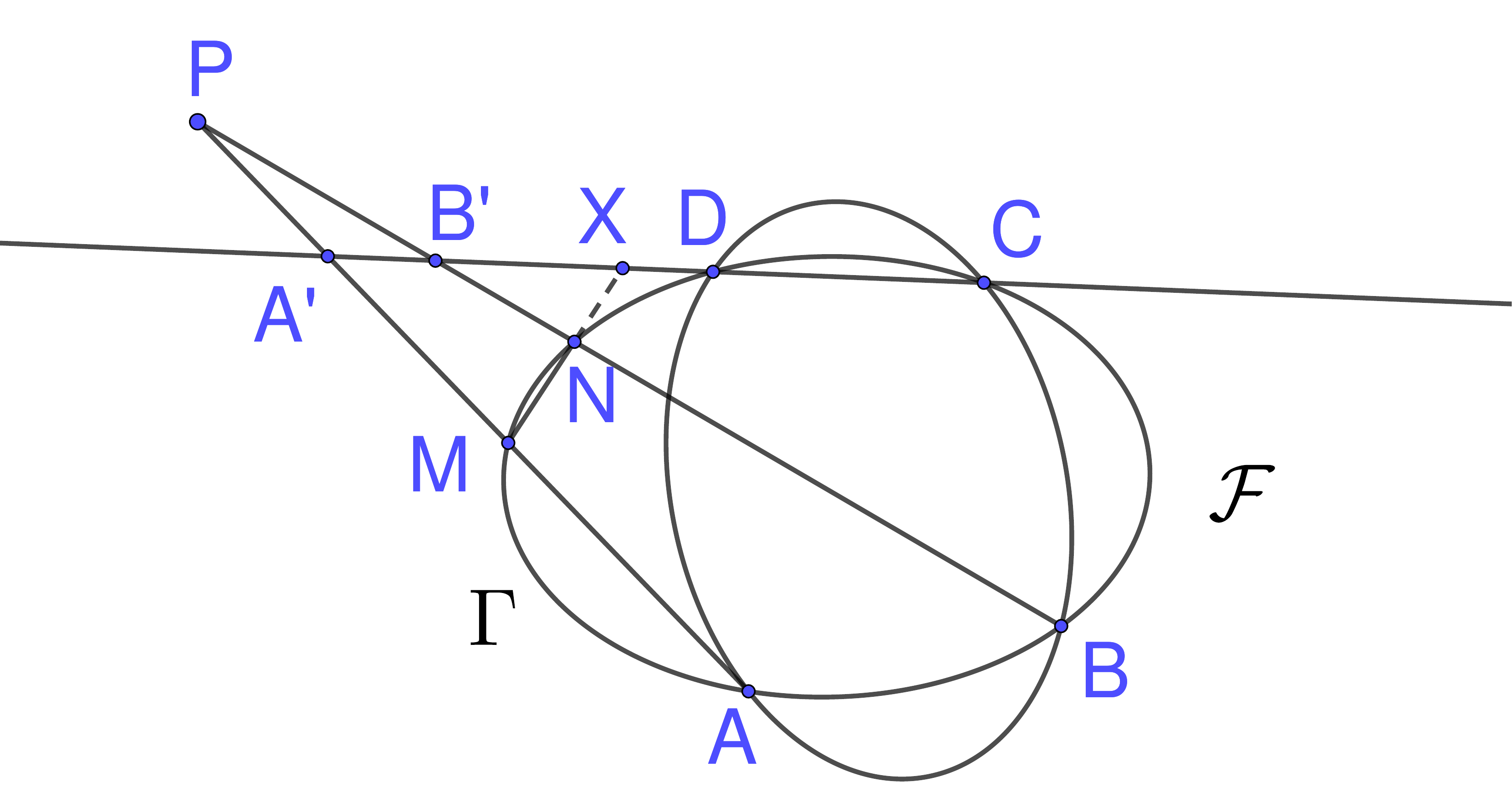}

\caption{Lemma \ref{lemma1}}
\label{lemma-image}
\end{figure}

\begin{proof}
Denote by $A'$ the intersection of $PA$ with $CD$, and the same for $B'$. Consider the pencil of conics $\mc{F}'$ defined by the base points $A, B, M, N$. By Proposition \ref{desargues-prop} it induces an involution $i$ on the line $CD$ (Desargues involution). 

The involution $i$ exchanges $C$ and $D$, and also exchanges $A'$ and $B'$. This already determines uniquely $i$. In particular, $i$ is independent of $\Gamma$, so independent of $M$ and $N$.

Finally, the intersection $X$ of $MN$ with $CD$ is the image under $i$ of $AB\cap CD$. Therefore $X$ is independent of $\Gamma \in \mc{F}$.
\end{proof}

\begin{proof}[Proof of Theorem \ref{thm3}]
Let $P$ be a point on $d$. Denote by $M_i$ and $N_i$ the intersections of $PA$ and $PB$ with $\Gamma_i$. According to the previous Lemma \ref{lemma1}, the lines $M_iN_i$ all pass through some fixed point $X$ on $CD$. Denote by $Q_i$ the intersection of $M_iN_i$ with $AB$, and by $P'$ the conjugate of $P$ with respect to $\mc{F}$. Finally we get

\begin{align*}
(M_1,M_2,M_3,M_4) &= (XM_1,XM_2,XM_3,XM_4) & \\
&= (Q_1,Q_2,Q_3,Q_4) & \text{ by Proposition \ref{cross-ratio-line}}\\
&= (P'Q_1,P'Q_2,P'Q_3,P'Q_4) \\
&= (\Gamma_1,\Gamma_2,\Gamma_3,\Gamma_4) & \text{ by Theorem \ref{thm1}}
\end{align*}

since $P'Q_i$ is the polar of $P$ with respect to $\Gamma_i$ (pass through $P'$ and $M_iN_i\cap AB$).
\end{proof}

\begin{figure}[h]
\centering
\includegraphics[height=4cm]{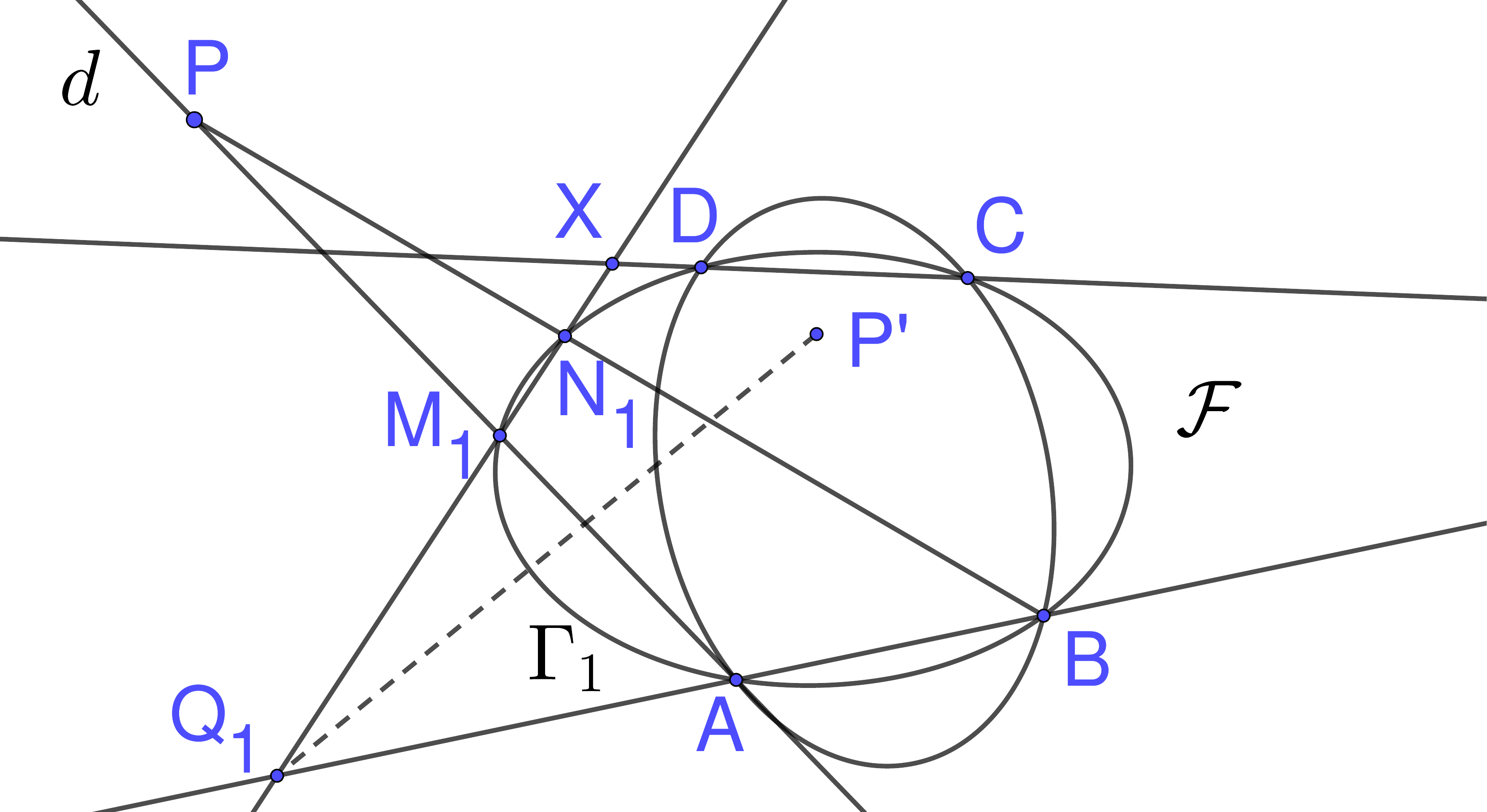}

\caption{Proof of Theorem \ref{thm3}}
\label{thm3-proof-image}
\end{figure}

The next characterization uses the following simple fact: a conic is uniquely determined by four points $A, B, C, D$ and a complex number $k$. It is the conic $\mc{C}$ passing through the four points such that $(A,B,C,D)_{\mc{C}}=k$.

\begin{thm}\label{thm4}
Denote by $k_i$ the complex number which uniquely determines a conic $\Gamma_i \in \mc{F}$ as above. Then $$(\Gamma_1,\Gamma_2,\Gamma_3,\Gamma_4)=(k_1,k_2,k_3,k_4).$$
\end{thm}

\begin{wrapfigure}{r}{0.4\textwidth}
   \includegraphics[width=0.38\textwidth]{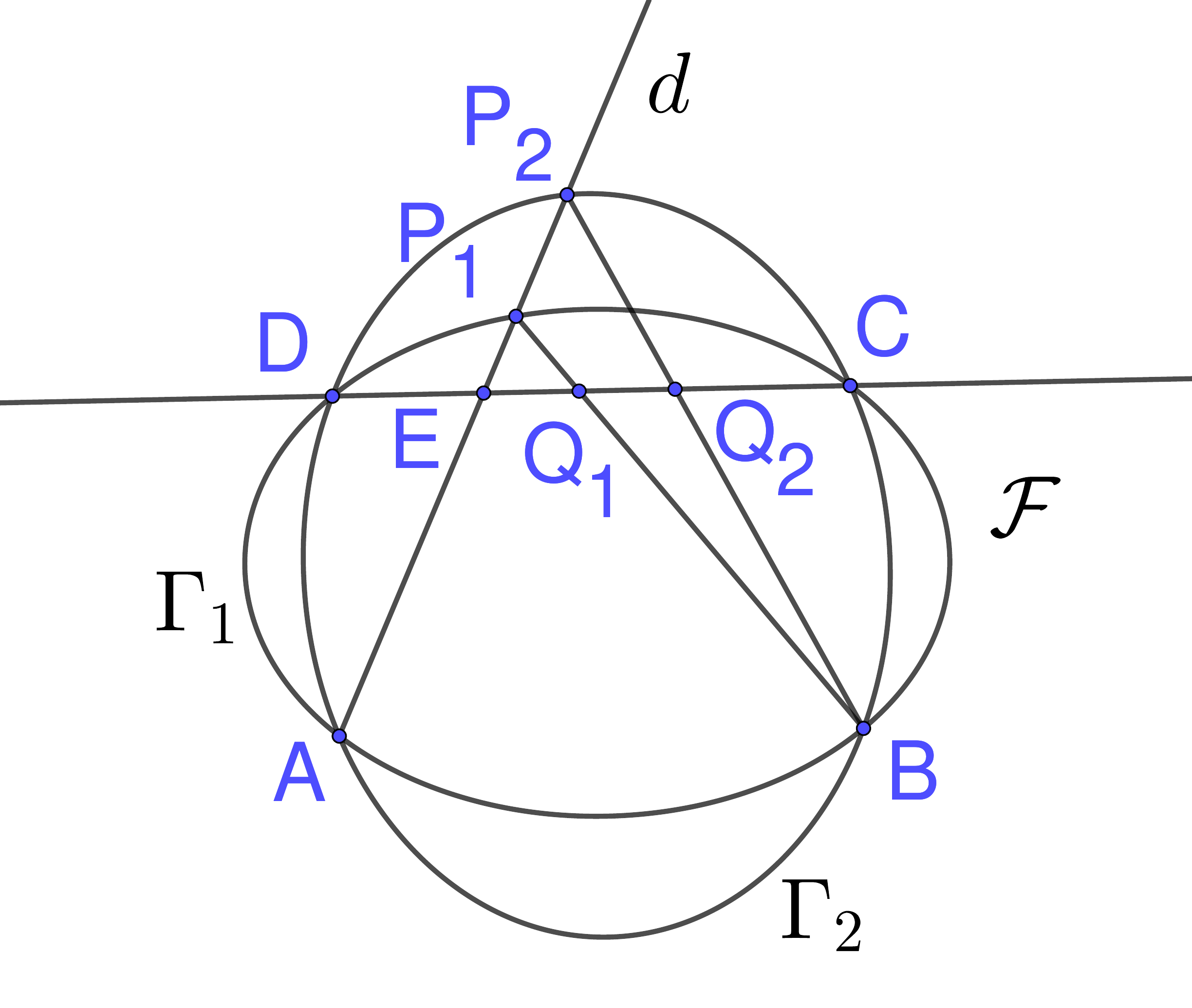}
 
\vspace{0.1cm}
\caption{Proof of Theorem \ref{thm4}}
\label{thm4-proof}
\end{wrapfigure}

\vspace{0.3cm}
\textit{Proof.}
Let $d$ be a line passing through $A$. Denote by $P_i$ the intersection of $d$ with $\Gamma_i$, by $Q_i$ the intersection of $BP_i$ with $CD$ and by $E$ the intersection of $d$ with $CD$, see Figure \ref{thm4-proof}.



Hence we have $k_i = (P_iA,P_iB,P_iC,P_iD) = (E,Q_i,C,D).$
So we have projected everything on the line $CD$. By a projective transformation, we can get $D \mapsto 0$, $E \mapsto 1$ and $C \mapsto \infty$. Denote the coordinate by $Q_i$ after this transformation by $q_i$. We then have by definition of the cross ratio
$k_i = (E,Q_i,C,D) = (1,q_i,\infty,0) = q_i.$
Hence we get $(k_1,k_2,k_3,k_4) = (Q_1,Q_2,Q_3,Q_4).$ By Proposition \ref{cross-ratio-line} we have $(Q_1,Q_2,Q_3,Q_4)=(P_1,P_2,P_3,P_4)$.
Therefore using Theorem \ref{thm3} we conclude:
\begin{center}
$(k_1,k_2,k_3,k_4) = (P_1,P_2,P_3,P_4) = (\Gamma_1,\Gamma_2,\Gamma_3,\Gamma_4). \qed$
\end{center}
\vspace{0.3cm}

\begin{Remark}
The following general statement is true: Given a pencil generated by two objects $P$ and $Q$, then any other object in the pencil is of the form $tP+(1-t)Q$. For four objects $O_i$ in the pencil, with parameters $t_i$, we have $$(O_1,O_2,O_3,O_4) = (t_1,t_2,t_3,t_4).$$
\end{Remark}

This remark directly implies that the abstract definition of the cross ratio $(\Gamma_1,\Gamma_2,\Gamma_3,\Gamma_4)$ (as points in a line in the space of quadratic forms, see the Introduction) coincides with our geometric characterizations.

\medskip
The cross ratio of four conics can be determined using another conic:

\begin{thm}\label{thm5}
Let $\mc{C}$ be a conic passing through two base points of the pencil of conics $\mc{F}$. The intersection points of $\mc{C}$ with $\Gamma_i$, other then the two base points, are denoted by $M_i$ and $N_i$. Then the lines $M_iN_i$ are concurrent and
$$(\Gamma_1,\Gamma_2,\Gamma_3,\Gamma_4)=(M_1N_1,M_2N_2,M_3N_3,M_4N_4).$$
\end{thm}

\vspace{-0.2cm}
\begin{figure}[h]
\centering
\includegraphics[height=3.7cm]{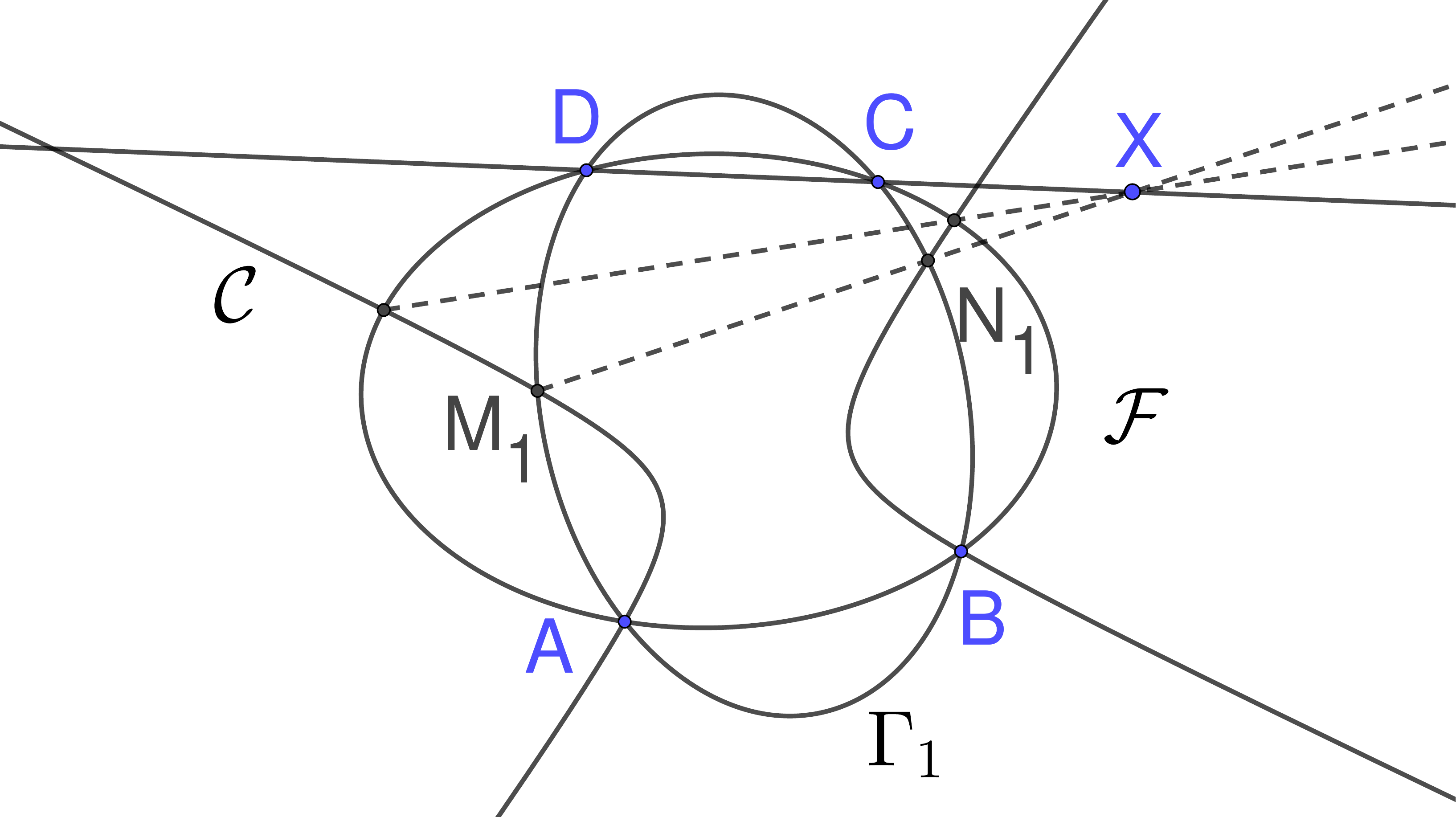}

\vspace{-0.2cm}
\caption{Cross ratio using a conic}
\label{thm5-proof-1}
\end{figure}

\begin{proof}
First, let us prove that the lines $M_iN_i$ all pass through a common point, which in addition lies on $CD$ (we assume that $\mc{C}$ passes through the base points $A$ and $B$). This is a generalization of lemma \ref{lemma1} and we can adopt its proof.

Consider the pencil of conics $\mc{F}'$ defined by the four base points $A, B, M_1, N_1$. On the line $CD$ this induces the Desargues involution $i$ which exchanges $C$ and $D$ and the two intersection points of $\mc{C}$ with $CD$. Thus $M_1N_1$ passes through the image under $i$ of $AB\cap CD$. This description is independent of $M_1N_1$.

Now we can show the second part:
Let $X'$ be the conjugated point to $X$ with respect to $\mc{F}$ (which also lies on $CD$), $p$ be the polar of $X$ with respect to $\mc{C}$ and $Q_i$ be the intersection of $p$ and $XM_i$ (see Figure \ref{thm5-proof-2}).

\begin{figure}[h]
\centering
\includegraphics[height=4cm]{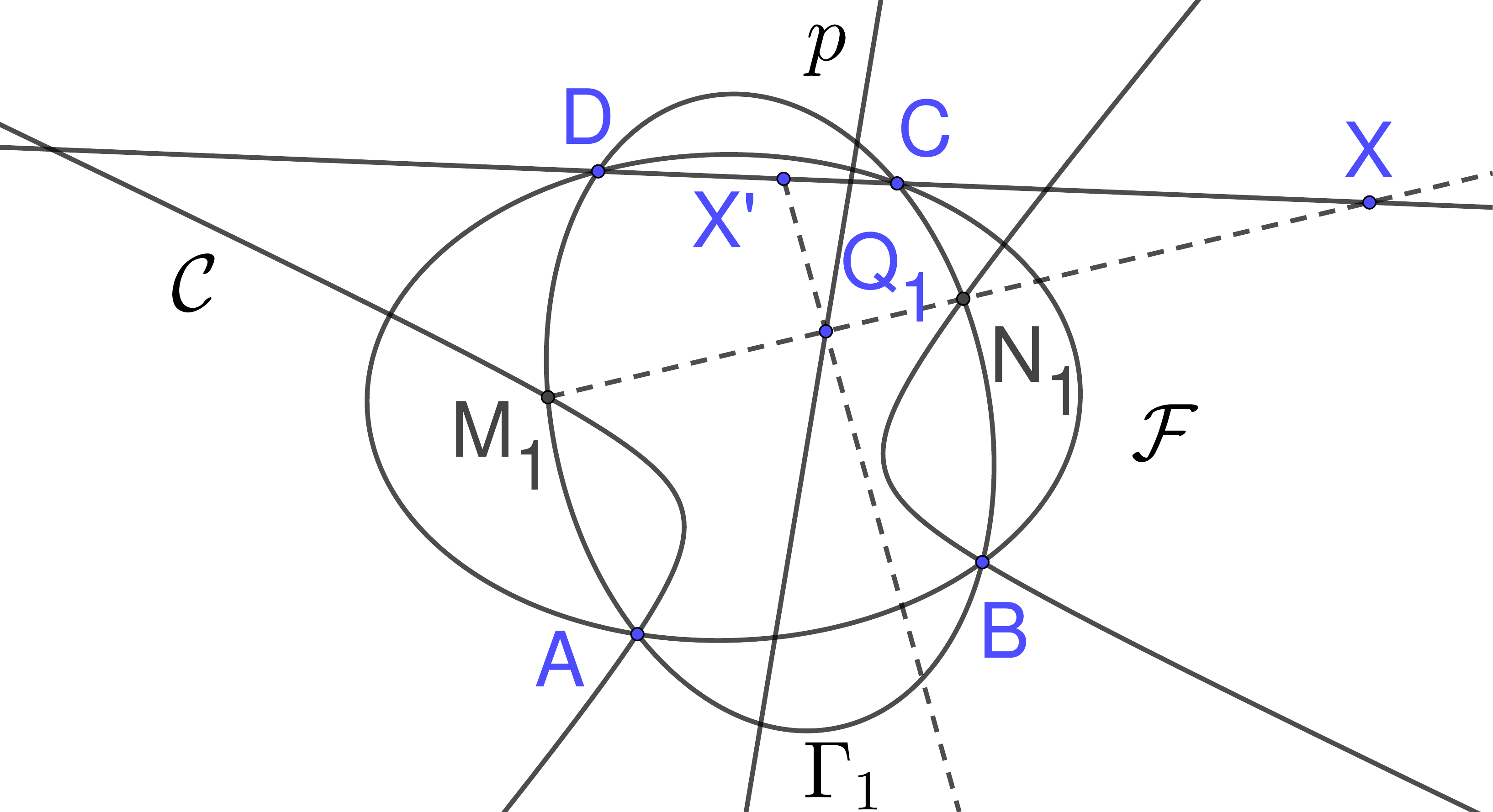}

\caption{Proof of Theorem \ref{thm5}}
\label{thm5-proof-2}
\end{figure}

According to the definition of the polar $p$, we have $$(X,Q_i,N_i,M_i)=-1.$$
Thus, the point $Q_i$ is also on the polar of $X$ with respect to $\Gamma_i$. Hence this polar is given by $X'Q_i$. Therefore
\begin{align*}
(M_1N_1,M_2N_2,M_3N_3,M_4N_4) &= (Q_1,Q_2,Q_3,Q_4) \\
&= (X'Q_1,X'Q_2,X'Q_3,X'Q_4) \\
&= (\Gamma_1,\Gamma_2,\Gamma_3,\Gamma_4)
\end{align*}
where in the last line we used Theorem \ref{thm1}.
\end{proof}

\begin{Remark}
The first part of the theorem, the fact that $M_iN_i$ goes through a fixed point on $CD$, is the ''three-conics theorem``. See for example \cite{Glaeser}, Theorem 7.3.2.
\end{Remark}

\begin{Remark}
This theorem gives another proof of theorem \ref{thm3} by taking for $\mc{C}$ a degenerate conic consisting of two lines going through two base points of $\mc{F}$.
\end{Remark}

\begin{Remark}
Together with Proposition \ref{conic-to-conic-prop}, we see that the pencil $\mc{F}$ induces an involution on the conic $\mc{C}$.
\end{Remark}

\begin{example}
Consider the following special case: take for $\mc{F}$ a pencil of circles and for $\mc{C}$ any circle, see Figure \ref{circle-pencil}. In fact, a conic is a circle iff it passes through the two cyclic points $I$ and $J$ with homogeneous coordinates $[1:i:0]$ and $[i:1:0]$ (which are complex and lie at infinity!), so the base points of a pencil of circles are the two intersection points $A$ and $B$, and the cyclic points $I$ and $J$. The circle $\mc{C}$ goes through $I$ and $J$.
Then Theorem \ref{thm5} gives directly the cross ratio of the circles.
\end{example}

\begin{figure}[h]
\centering
\includegraphics[height=3.7cm]{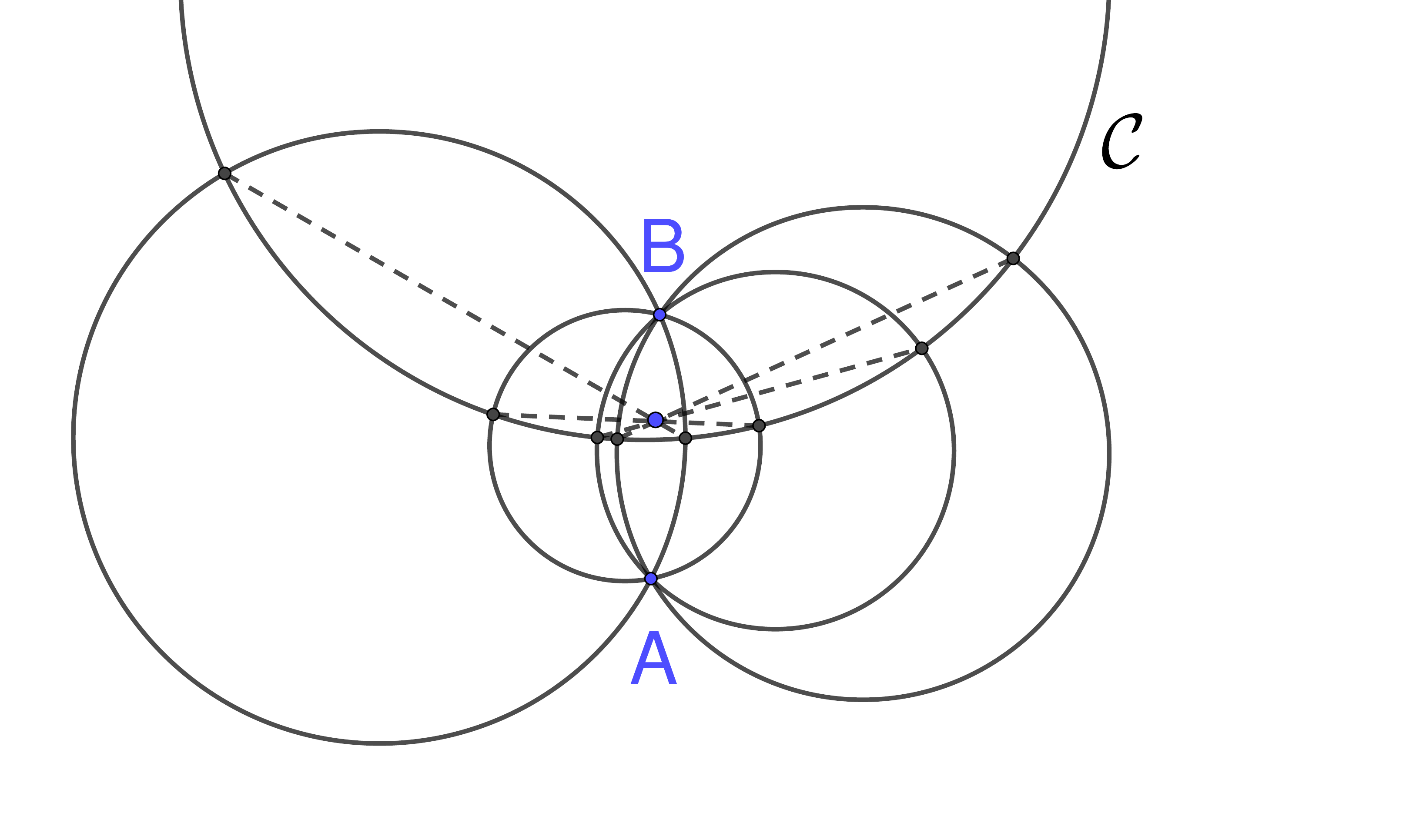}

\vspace{-0.3cm}
\caption{Case of a pencil of circles}
\label{circle-pencil}
\end{figure}

As corollary, we can consider the special case when the conic $\mc{C}$ goes through three of the base points. This is our favorite geometric way to get the cross ratio of conics:

\begin{coro}\label{thm6}
Let $\mc{C}$ be a conic going through three base points of the pencil $\mc{F}$. Denote by $M_i$ the fourth intersection point of $\mc{C}$ with $\Gamma_i$. Then $$(\Gamma_1,\Gamma_2,\Gamma_3,\Gamma_4) = (M_1,M_2,M_3,M_4)_{\mc{C}}.$$
\end{coro}

\begin{figure}[h]
\centering
\includegraphics[height=4.2cm]{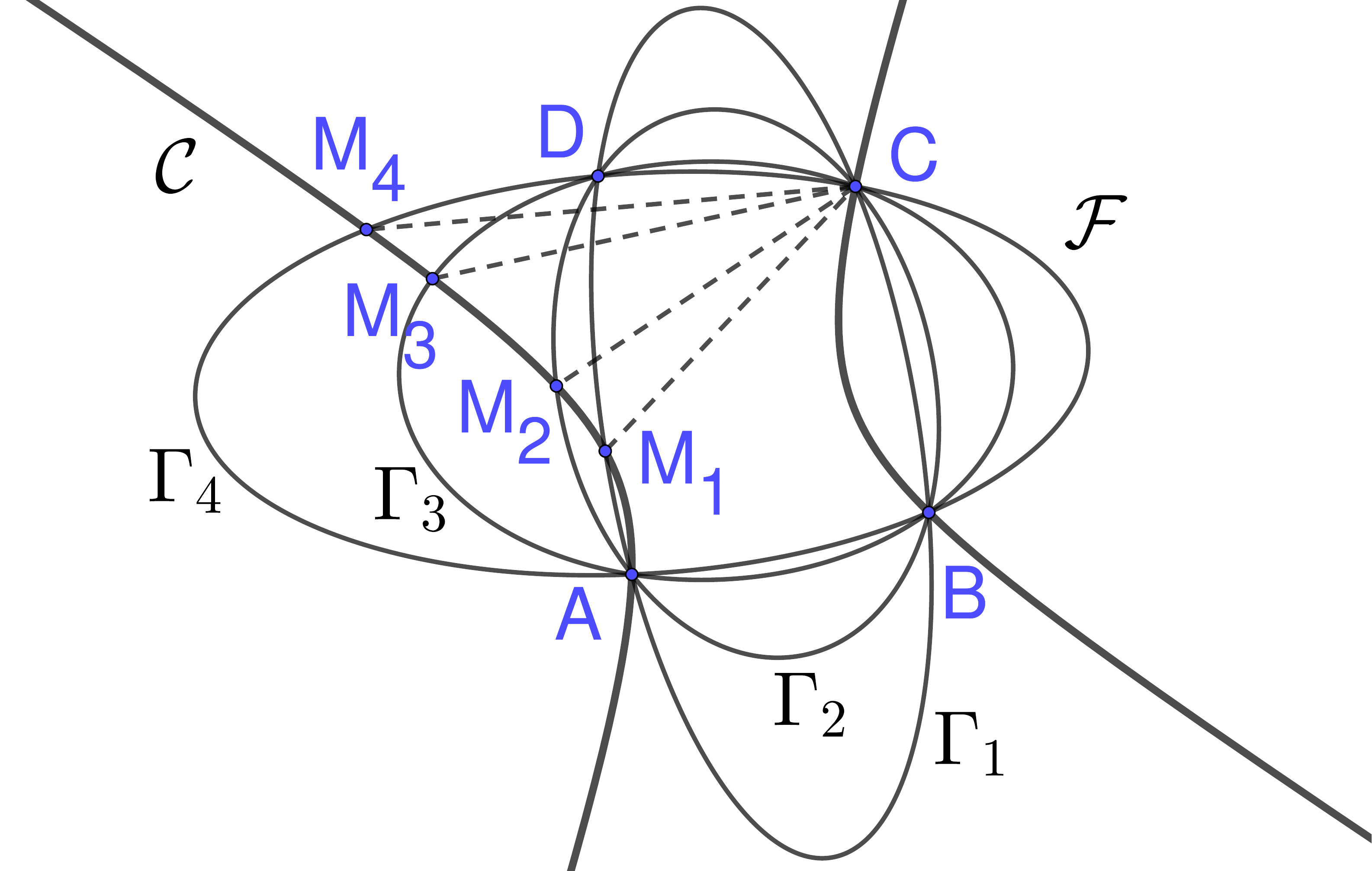}

\caption{Characterization with a conic}
\label{thm6-image}
\end{figure}

\begin{proof}
Take Theorem \ref{thm5} with $X=N_1=N_2=N_3=N_4=C$. Then using Proposition \ref{cross-ratio-conic} we conclude: $$(\Gamma_1,\Gamma_2,\Gamma_3,\Gamma_4) = (CM_1,CM_2,CM_3,CM_4) = (M_1,M_2,M_3,M_4)_{\mc{C}}.$$
\end{proof}

To finish, we prove that the conjugation with respect to a pencil of conics also preserves the cross ratio of lines in some sense:
\begin{thm}\label{thm7}
Let $d_1, d_2, d_3, d_4$ four concurrent lines, passing through some point $P$. Denote by $\mc{C}_i$ the conic which is the image of $d_i$ by conjugation with respect to some pencil $\mc{F}$. Then the $\mc{C}_i$ form a pencil of conics and $$(d_1,d_2,d_3,d_4)=(\mc{C}_1,\mc{C}_2,\mc{C}_3,\mc{C}_4).$$
\end{thm}

\begin{proof}
First, all $\mc{C}_i$ pass through the three double points of $\mc{F}$ and through $P'$. Hence they all belong to a common pencil.

Second, let $l$ denote a line intersecting $d_i$ in some point $Q_i$ (see Figure \ref{thm6-image}). The image of $l$ under conjugation with respect to $\mc{F}$ is a conic $\mc{C}$ passing through the double points $R, S, T$ of $\mc{F}$. Denote by $T_i$ the fourth intersection point of $\mc{C}$ with $\mc{C_i}$ (other than $R, S, T$). We have that $T_i=Q_i'$ since $Q_i=d_i\cap l$.

\begin{figure}[h]
\centering
\includegraphics[height=4.2cm]{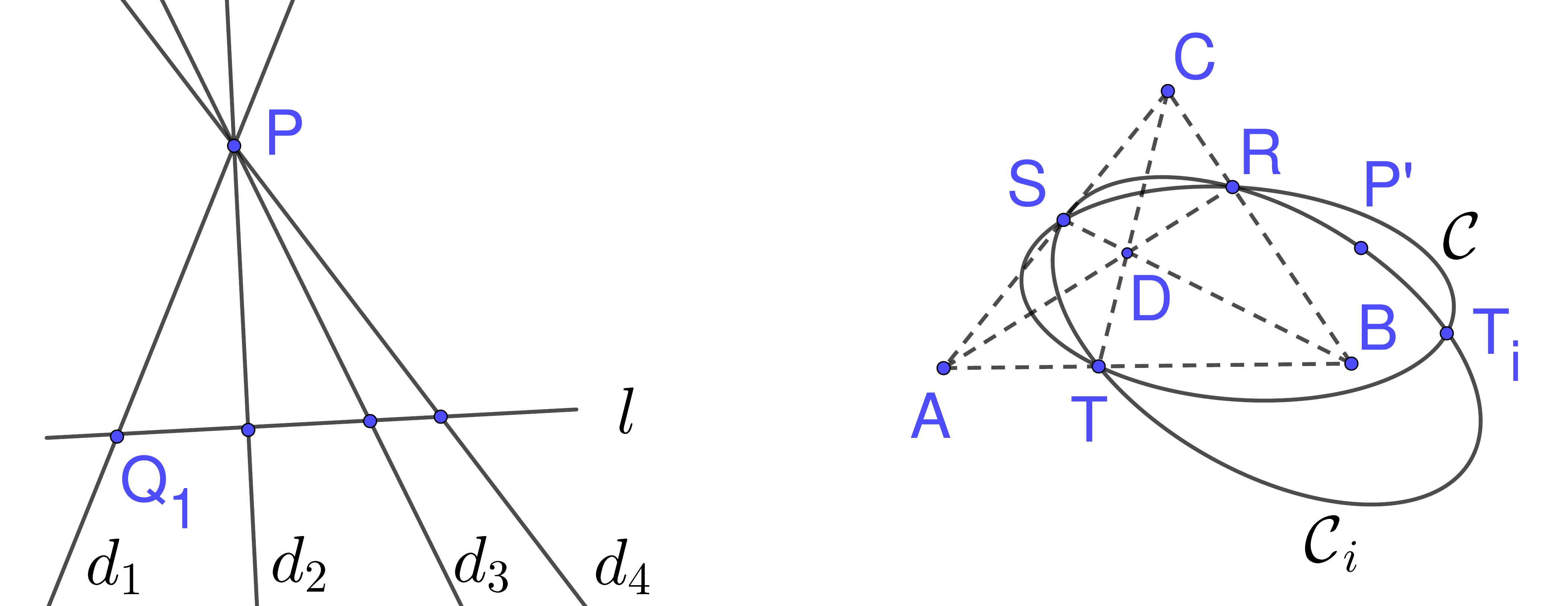}

\caption{Proof of Theorem \ref{thm7}}
\label{thm7-proof}
\end{figure}

Therefore
\begin{align*}
(\mc{C}_1,\mc{C}_2,\mc{C}_3,\mc{C}_4) &= (T_1,T_2,T_3,T_4)_{\mc{C}} & \text{ by Corollary \ref{thm6}} \\
&= (Q_1,Q_2,Q_3,Q_4) & \text{ by Proposition \ref{conjugation-pencil}} \\
&= (d_1,d_2,d_3,d_4) & \text{ by Proposition \ref{cross-ratio-line}}
\end{align*}
\end{proof}

As final remark, we wish to share our impression with the reader that \textit{in projective geometry, all reasonable possible equalities between cross ratios are true}.

\bigskip
\textit{\textbf{Acknowledgments.}} This project has been carried out during my undergraduate studies in Paris in 2012-2013. I am deeply grateful to Bodo Lass for his passionate lectures on projective geometry. Only during my postdoc at the Max-Planck Institute Bonn I took the time to write these notes down. I'm grateful to the institute for their support.

\end{document}